\documentclass{birkjour_t2}
\usepackage{hyperref}
\usepackage{subcaption}

\newtheorem{theorem}{Theorem}

\newtheorem{remark}{Remark}

\newtheorem{lemma}{Lemma}
\newtheorem{corollary}{Corollary}

\author[Broadbridge]{Philip  Broadbridge}

\address{%
La Trobe University,\\
Melbourne,\\
Australia,}
\email{p.broadbridge@latrobe.edu.au}

\thanks{Philip Broadbridge, La Trobe University, Melbourne, Australia \href{p.broadbridge@latrobe.edu.au}{p.broadbridge@latrobe.edu.au}  \\
Illia Donhauzer, corresponding author, La Trobe University, Melbourne, Australia \href{i.donhauzer@latrobe.edu.au}{i.donhauzer@latrobe.edu.au} \\
Andriy Olenko, La Trobe University, Melbourne, Australia \href{a.olenko@latrobe.edu.au}{a.olenko@latrobe.edu.au} \\
This research was supported under the Australian Research Council's Discovery Projects funding scheme (project number  DP220101680).  I.Donhauzer and A.Olenko also would like to thank for partial support provided by the La Trobe SEMS CaRE grant.}
\author[Donhauzer]{Illia Donhauzer}
\address{Corresponding author,\\
La Trobe University,\\
Melbourne,\\
Australia,}
\email{i.donhauzer@latrobe.edu.au}
\author[Olenko]{Andriy Olenko}
\address{%
La Trobe University,\\
Melbourne,\\
Australia,}
\email{a.olenko@latrobe.edu.au}

\title{Stochastic diffusion within expanding space-time}

\begin{document}
\maketitle

\textbf{Abstract} The paper examines stochastic diffusion within an expanding space-time framework. It starts with providing a rationale for the considered model and its motivation from cosmology where the expansion of space-time is used in modelling various phenomena. Contrary to other results in the literature, the considered in this paper general stochastic model takes into consideration the expansion of space-time. It leads to a stochastic diffusion equations  with coefficients that are non-constant and evolve with the expansion factor. Then, the Cauchy problem with random initial conditions is posed and investigated. The exact  solution to a stochastic diffusion equation on the expanding sphere is derived. Various probabilistic properties of the solution are studied, including its dependence structure, evolution of the angular power spectrum and local properties of the solution and its approximations by finite truncations. The paper also characterises the extremal behaviour of the random solution by establishing upper bounds on the probabilities of large deviations. Numerical studies are undertaken to illustrate the obtained theoretical results and demonstrate the evolution of the random solution.



\section{Introduction}

The NASA mission WMAP and the ESA (European Space Agency) mission Planck have been instrumental in collecting highly accurate cosmological data, resulting in a precise map depicting the distribution of Cosmic Microwave Background Radiation (CMB)  \cite{adam2016planck, ade2016planck}. It is expected that new experiments, in particular, within CMB-S4 Collaboration and  ESA's Euclid mission,  will provide measurements of the CMB at unprecedented precision.

The CMB spectrum indicates that since the last scattering around 380,000 years after the big bang, the universe has been transparent to electromagnetic radiation. However the universe is not transparent to charged cosmic ray particles. They are deviated by magnetic fields as they pass close to galaxies. Recent estimates of the number of galaxies in the observable universe range from $2\times10^{12}$ to $6\times10^{12}$. In any angular aperture of observation, there will be part of at least one galaxy. Extragalactic cosmic ray particles reach the earth in a small number of showers each year. They are distinguished from local galactic cosmic rays by their high particle energy, greater than $5 \times 10^{18} eV$. This corresponds to charged particles typically arriving at more than half the speed of light. They have likely been travelling for a very long time in cosmic terms, during which they would have been deviated by a number of galaxies. This results in a long-term diffusive redistribution of matter throughout the universe.

This dynamical process is very complicated. Effective diffusion occurs partly by scattering due to magnetic fields \cite{han2017observing} and also by motion of the magnetic field lines themselves \cite{jokipii1969stochastic}. It has also been argued that the temperature gradients formed from galaxies can effectively repel particles,  enough to avoid gravitational capture \cite{shtykovskiy2010thermal}.

Due to the long distances between galaxy clusters, scattering is intermittent, reasonably described by a fractional $\alpha$-stable L\'evy distribution with tail probabilities of order $|x|^{-\alpha}$ for exceeding a large displacement, and the distribution due to a local disturbance broadening asymptotically in proportion to $t^{1/\alpha}.$ Such a distribution can result from a fractional super-diffusion that is of order $\alpha<2$. Data from intra-galactic cosmic rays evidence values of $\alpha$ less than 0.5 \cite{buonocore2021anomalous}.

The cosmological missions have yielded high resolution maps of CMB. The necessity of modelling and analysing them have recently attracted increasing attention to the theory of spherical random fields. From the mathematical point of view and for modelling purposes, a map depicting the distribution of CMB can be regarded as a single realization of a random field on a sphere of a large radius. The sphere plays a role of the underlying space and expands in time. For the Dark-energy-dominated era, the expansion factor has the exponential form \cite{broadbridge2020solution}. This expansion impacts the stochastic diffusion and should be incorporated in a model for the evolution of CMB. Contrary to the other models in the literature, this paper takes into consideration the expansion of space-time, which leads to a stochastic diffusion equation with coefficients that are non-constant and evolve with the expansion factor.

We refer to the monograph \cite{marinucci2011random} for the systematic exposition of the main results of the theory of spherical random fields. The paper \cite{lang2015isotropic} studied isotropic random fields on high-dimensional spheres $\mathbb{S}^n$ and established  connections between the smoothness of the covariance kernel and the decay of their angular power spectrum, derived conditions for almost sure sample continuity and sample differentiability of spherical random fields, and obtained sufficient conditions for their $L_2$ continuity in terms of the decay of the angular power spectrum.

Another important direction of the modern theory of random fields is the exploration of extremes of random fields including fields given on manifolds, see the classical results in \cite{talagrand},  \cite{dudley1967sizes},  \cite{piterbarg1996asymptotic} and their modern generalisations. The monograph \cite{buldygin2000metric} examined extremes of sub-Gaussian fields while the expected Euler characteristic method developed by Adler and Taylor was demonstrated in \cite{adler2007random}. The publication \cite{cheng2016excursion} provided the asymptotics of excursion probabilities for both smooth and non-smooth  spherical Gaussian fields. Some inequalities for excursion probabilities for spherical sub-Gaussian random fields were obtained in \cite{sakhno2023estimates}. Another approach utilized limit theorems for sojourn measures, see \cite{leonenko_ruiz-medina_2023} and \cite{makogin}.

Stochastic partial differential equations (SPDE) are the main tool to model the evolution of spherical random fields over time. They have been extensively studied, see, for example, \cite{d2014coordinates}, \cite{lang2015isotropic}, \cite{orsingher1987stochastic} and the references therein. Several models were recently presented in \cite{anh2018approximation, broadbridge2019random, broadbridge2020spherically, LOV} that employed stochastic hyperbolic diffusion equations and modelled various types of evolution of spherical random fields.

This article integrates the aforementioned approaches and extend them to the context of SPDEs within an expanding space-time. The equations studied in this paper differ from the mentioned SPDEs and are  given as a hyperbolic diffusion on the expanding sphere. Motivated by cosmological applications, the exponential expansion is used in the considered model, which leads to a stochastic diffusion equation with non-constant coefficients.

The article also conducts the numerical analysis of the solution of the considered model. The numerical analysis section investigates the evolution of the solution of the studied SPDE over space-time, the structure of its  space-time dependencies, and its extremes. The CMB intensity map from the mission Planck and its spectrum  are used as initial conditions.  The numerical analysis confirms and visualizes the obtained theoretical results.

The main novelties of the paper include:
\begin{itemize}
    \item    The consideration of diffusion within an expanding space-time framework;
   \item     The examination of equations with non-constant coefficients dependent on the expansion factor values;
   \item     An exploration of both the local and asymptotic properties of the solutions and their respective approximations;
   \item     An analysis of excursion probabilities associated with the solutions and their approximations.
\end{itemize}

The paper is structured as follows: Section~\ref{sec2} provides  the main definitions and notations. Section~\ref{sec3} presents the diffusion model within an expanding time-space universe. Then, this section investigates  the initial-value problem for this diffusion equation and explores the properties of the derived non-random solutions. Section~\ref{sec3} is dedicated to the examination of the equation with random initial conditions. The solution to the equation and its associated covariance function are derived. Section~\ref{sec:approx} investigates  properties of stochastic solutions and their corresponding approximations. Sections~\ref{sec_excur} studies excursion probabilities related to the solutions and their approximations. Finally, Section~\ref{sec_num} presents simulation studies that illustrate the properties of the solutions and the obtained results.

\section{Main definitions and notations}\label{sec2}
This section reviews the main definitions and notations used in this paper and provides required background knowledge from the theory of random fields.

$C$ with subindices represents a generic finite positive constant. The values of constants are not necessarily the same in each appearance and may be changed depending on the expression.  $||\cdot||$ is the Euclidean norm in $\mathbb{R}^3,$ $Leb(\cdot)$ stands for the Lebesgue measure on the unit sphere $ \mathbb{S}^2 = \{\textbf{\textit{x}}: \ ||\textbf{\textit{x}}|| = 1,\ \textbf{\textit{x}}\in\mathbb{R}^3\}.$

The notation $\textit{\textbf{x}} = (x_1,x_2,x_{3})\in \mathbb{S}^2$ is used to define Euclidean
coordinates of points, while the notation $(\theta, \varphi), \ 0\leq\theta\leq\pi, \ 0\leq\varphi<2\pi,$ is used for the corresponding spherical coordinates.  They are related by the following transformations
\begin{align*}\label{spher_coord}
\begin{split}
&x_1  = \cos\theta_1, \\
&x_2 = \sin\theta_1\cos\theta_2, \\
&x_3 = \sin\theta_1\sin\theta_2.
\end{split}
\end{align*} In what follows, $ \Theta$ denotes the angular distance between two points  with spherical coordinates $(\theta, \varphi)$ and $(\theta', \varphi')$  on the unit sphere $\mathbb{S}^2.$

By $L^2(\mathbb{S}^2)$ we denote the Hilbert space of square integrable functions on $\mathbb{S}^2$ with the following canonical inner product \cite[Page 8]{atkinson2012spherical} \[ \langle   f,g \rangle_{L^2(\mathbb{S}^2)} = \int\limits_0^\pi\int\limits_0^{2\pi} f(\theta,\varphi)g^*(\theta,\varphi)\sin\theta d\theta d \varphi,\ \ \ f,g\in L^2(\mathbb{S}^2), \] where $g^*(\cdot)$ denotes a complex conjugate of $g(\cdot),$ and the induced  norm
\[ ||f||^2_{L^2(\mathbb{S}^2)} = \int\limits_0^\pi\int\limits_0^{2\pi} |f(\theta,\varphi)|^2\sin\theta d\theta d \varphi.\]

The spherical harmonics $Y_{lm}(\theta,\varphi), \ l=0,1,..., m = -l,..l,$  are given as
\[ Y_{lm}(\theta,\varphi) = d_{lm}\exp\{ im\varphi \}P_l^m(\cos\theta), \] where
\[d_{lm} = (-1)^m\sqrt{\frac{2l+1}{4\pi}\frac{(l-m)!}{(l+m)!}},\]  $P_l^m(\cdot)$ are the associated Legendre polynomials with the indices $l$ and $m,$ and $P_l(\cdot)$ is the $l$th Legendre polynomial
\[ P_l^m(x) = (-1)^m(1-x^2)^{\frac{m}{2}}\frac{d^m}{dx^m}P_l(x), \ \ \ P_l(x)=\frac{1}{2^ll!}\frac{d^l}{dx^l}(x^2-1)^l. \]

The spherical harmonics $Y_{lm}(\theta,\varphi), \ l=0,1,..., m = -l,..l,$ form an orthogonal basis in the Hilbert space $L^2(\mathbb{S}^2),$  i.e.
\[ \langle   Y_{lm},Y_{l'm'} \rangle_{L^2(\mathbb{S}^2)} = \delta_{ll'}\delta_{mm'},\] where $\delta_{ll'}$ is the Kronecker delta function.

The addition formula for the spherical harmonics states that
\begin{equation}\label{add_form} \sum_{m=-l}^lY_{lm}(\theta, \varphi) Y^*_{lm}(\theta', \varphi') = \frac{2l+1}{4\pi}P_l(\cos \Theta).
\end{equation}

 For any $f\in L^2(\mathbb{S}^2)$ it holds
\[ f(\theta,\varphi) = \sum_{l=0}^\infty\sum_{m=-l}^l f_{lm}Y_{lm}(\theta,\varphi),\] where $f_{lm}\in \mathbb{C}, \ l=0,1,...,\ m=-l,...,l.$

A spherical random field $T(\theta,\varphi),$  is a collection of random variables given on  a common complete probability space $\{\Omega, \mathfrak{F}, P  \}$ and indexed by parameters $\theta,\varphi.$ In this paper, we consider real-valued spherical random fields that are continuous and twice-differentiable with probability $1.$

By $L^2(\Omega\times\mathbb{S}^2)$ we denote the Hilbert space of spherical random fields that have a finite norm
\[ ||T||_{L^2(\Omega\times\mathbb{S}^2)}= E\left( \int\limits_{0}^{\pi}\int\limits_{0}^{2\pi} T^2(\theta,\varphi)  \sin\theta d\varphi d\theta \right).\]

 A spherical random field $T(\theta,\varphi)$ is called isotropic if its finite-dimensional distributions are invariant with respect to rotation transformations, i.e. if
\[P(T(\textit{\textbf{x}}_1)<a_1,...,T(\textit{\textbf{x}}_k)<a_k) =  P(T(A\textit{\textbf{x}}_1)<a_1,...,T(A\textit{\textbf{x}}_1)<a_k)\] for any $\textit{\textbf{x}}_i\in\mathbb{S}^2,$ $a_i\in\mathbb{R},\ i=1,2,...k,\ k\in\mathbb{N}$ and any rotation matrix $A.$

An isotropic spherical random field $T(\theta,\varphi) \in L^2(\Omega\times\mathbb{S}^2)$ allows a representation as the following Laplace series~{\rm \cite{lang2015isotropic}}
\begin{equation}\label{eq:repres} T(\theta, \varphi) = \sum_{l=0}^\infty\sum_{m=-l}^l a_{lm} Y_{lm}(\theta,\varphi),\end{equation} where the convergence is in the space $L^2(\Omega\times\mathbb{S}^2),$ i.e..
\[ \lim\limits_{L\to\infty} E\left( \int\limits_{0}^{\pi}\int\limits_{0}^{2\pi} \left(T(\theta,\varphi) - \sum_{l=0}^L\sum_{m=-l}^l a_{lm} Y_{lm}(\theta,\varphi)\right)^2  \sin\theta d\varphi d\theta \right)=0.\] The random variables $a_{lm}, \ l=0,1,...,\ m=-l,...,l,$ are given by  the next stochastic integrals defined in the mean-square sense

\begin{equation}\label{eq:alm} a_{lm} = \int\limits_{0}^\pi \int\limits_{0}^{2\pi}T(\theta,\varphi)Y^*(\theta,\varphi)\sin\theta d\theta d\varphi.\end{equation}

Note that a spherical random field $T(\theta, \varphi)$ takes a constant random value, i.e. $T(\theta, \varphi) = \xi,\ \xi\in L^2(\Omega), \forall \theta, \varphi,$ with probability 1, if and only if $a_{lm}=0$ a.s. for $l=1,2,...$ Indeed, by the properties of spherical harmonics, for $T(\theta, \varphi) = \xi$ the
integrals in \eqref{eq:alm} equal to $0$ a.s.  for $l=1,2,...$ The inverse statement is trivial due to \eqref{eq:repres}.

If $T(\theta,\varphi)\in L^2(\Omega\times\mathbb{S}^2)$ is a centered real-valued Gaussian random field, then it allows the representation \eqref{eq:repres}, where $a_{lm},\ l = 0,1,..., \ m=-l,...,l,$ are Gaussian random variables such that
\begin{equation*}
a_{lm} = (-1)^ma_{l(-m)},
\end{equation*}
\begin{equation*}
Ea_{lm}=0, \ Ea_{lm}a_{l'm'}^* = \delta_m^{m'}\delta_l^{l'}C_l.
\end{equation*}  The sequence $\{ C_l, l=0,1,...\}$ is called the angular power spectrum of the isotropic random field $T(\theta,\varphi).$ The series \eqref{eq:repres} converges in the $L_2(\Omega\times\mathbb{S}^2)$ sense if it holds true  {\rm \cite[Section 2]{lang2015isotropic}} that
\begin{equation}\label{eq:ang_conv} \sum\limits_{l=0}^\infty C_l(2l+1)<+\infty. \end{equation}
We assume that condition (\ref{eq:ang_conv}) remains valid  throughout all subsequent sections of the paper.

The Bessel function of the first kind $J_\nu(x), \ x\geq0,\ \nu\in\mathbb{R},$ is given by the following series~\cite[9.1.10]{abramowitz1968handbook}
\[J_\nu(x) = \left( \frac{1}{2}x \right)^\nu \sum\limits_{k=0}^\infty \frac{\left( -\frac{1}{4}x^2 \right)^k}{k! \Gamma(\nu+k+1)}.\] It has a finite value at the origin if $\nu\geq 0$ and a singularity if $\nu<0.$

For a noninteger $\nu,$ the Bessel function of the second kind $Y_\nu(x), \ x\geq0,$ is given as~\cite[9.1.2]{abramowitz1968handbook}
\[ Y_\nu(x) = \frac{J_\nu(x)\cos(\nu\pi)-J_{-\nu}(x)}{\sin(\nu\pi)}.\] The Bessel function of the second kind $Y_n(x), \ x\geq0,$ of a positive integer order $n\in\mathbb{N}$ is obtained as the limit $Y_n(x) = \lim\limits_{\nu\to n}Y_\nu(x),$ see \cite[9.1.11]{abramowitz1968handbook}.

\section{Spherical diffusion in expanding space-time}
\label{sec3}

This section provides a justification for the model and obtain certain properties of solutions to the non-random version of the considered diffusion equations.

The universe is observed to have spatial cross sections with zero curvature. The Friedmann–Lema\^{i}tre–\\Robertson–Walker metric (FLRW) on the sphere of radius $r$ is
\[ ds^2 = c^2 dt^2 - a^2(t)(r^2d\theta^2 + r^2\sin^2\theta d\varphi^2),\] in the spherical space coordinates $(\theta, \varphi)$ and the cosmic time coordinate $t,$ where the term $a(\cdot)$ is the expansion factor. Note that for the Dark-energy-dominated era \cite{hill2018formal,hill2019some}, the expansion factor has the exponential form of the maximally symmetric de Sitter universe, $a(t) = e^{\sqrt{\frac{\Lambda}{3}}ct}$ \cite{broadbridge2020solution,Weinberg}.

The spherical diffusion is given by the following equation

\begin{equation}\label{eq} \frac{1}{D} \frac{\partial \widetilde{u}}{\partial t} + g^{\mu \nu} \nabla_\mu \nabla_{\nu} \widetilde{u} =0,
\end{equation}  where $\nabla_\mu$ is the covariant derivative operator and $g^{\mu \nu}$ are the elements of the contravariant metric tensor where the indices take values from 0 to 2. We also impose the following initial conditions
\begin{equation}\label{cond_t}
\widetilde{u}(t, \theta, \varphi)|_{t = 0} = \delta(\theta, \varphi), \ \ \ \frac{\partial \widetilde{u}(t, \theta, \varphi)}{\partial t}{\biggl |}_{t = 0} =0.
\end{equation}

The covariant derivatives do not commute when they act on vectors, see \cite{broadbridge2023non}.
However, the Laplace-Beltrami operator on the sphere can be expanded unambiguously in terms of partial derivatives as
\begin{equation} \label{LB}
g^{\mu \nu}\nabla_\mu\nabla_\nu \widetilde{u} = \frac{1}{\sqrt{|g|}}\partial_\mu\sqrt{|g|}g^{\mu\nu} \partial_\nu \widetilde{u},
\end{equation} where $g$ is the covariant metric tensor. For the above defined FLRW metric, $g$ is the diagonal matrix with the elements $c^{2}, -r^2a^{2}(t), -r^2a^{2}(t)sin^{2}\theta,$ and the contravariant metric tensor is the inverse of $g.$

Changing to a conformal time coordinate
\[ \eta = \int\limits_0^t\frac{1}{a(s)}ds = \eta_\infty\big(1-e^{- t/\eta_\infty}\big),\] the FLRW metric becomes
\[ ds^2 = e^2(\eta)(c^2 d\eta^2 -r^2d\theta^2 - r^2\sin^2\theta d\varphi^2),\] where $e(\eta):=\frac{\eta_\infty}{\eta_\infty-\eta} = a(t), \ \eta_\infty = \frac{1}{c}\sqrt{\frac{3}{\Lambda}},$ and the covariant metric tensor $g$ becomes the diagonal matrix with the elements $c^{2}e^2(\eta), -r^2e^2(\eta), -r^2e^2(\eta)\sin^{2}\theta.$

Using \eqref{LB} and the above expression of the covariant metric tensor in terms of the conformal time, the equation \eqref{eq} can be written in the following form

\begin{equation} \label{eq_main}  \left(\frac{e(\eta)}{D} + \frac{e'}{c^2e}\right)\frac{\partial \widetilde{u}}{\partial \eta} + \frac{1}{c^2} \frac{\partial^2 \widetilde{u}}{\partial \eta^2} - \frac{1}{r^2} \frac{\partial^2\widetilde{u}}{\partial \theta^2} - \frac{1}{r^2 \sin^2\theta}\frac{\partial^2\widetilde{u}}{\partial \varphi^2} -\frac{\cot\theta}{r^2}\frac{\partial \widetilde{u}}{\partial \theta} = 0.
\end{equation} As  for $t = 0$ it holds that $\eta = 0$ and $\frac{d\eta}{dt}\big|_{t=0} = 1,$ the initial conditions \eqref{cond_t} become

\begin{equation}
\label{cond}
\widetilde{u}(\eta, \theta, \varphi)|_{\eta = 0} = \delta(\theta, \varphi), \ \ \ \frac{\partial \widetilde{u}(\eta, \theta, \varphi)}{\partial \eta}{\biggl |}_{\eta = 0} =0.
\end{equation} The equation \eqref{eq_main} is different compared to the models studied in the literature, see \cite{anh2018approximation, broadbridge2019random, broadbridge2020spherically, LOV} and the references therein. Namely,  for the underlying FLRW space-time metric  the coefficient of the term $\frac{\partial \widetilde{u}}{\partial \eta}$  in  \eqref{eq_main} is not constant and depends on the evolution of the expansion factor~$e(\eta).$

\begin{theorem}
\label{th1}
The solution $\widetilde{u}(\eta, \theta, \varphi)$ of the equation \eqref{eq_main} with the initial conditions \eqref{cond} is given by the following series

\[\sum_{l=0}^{\infty} F_l(\eta) \sum_{m=-l}^l Y_{lm}^*(\textbf{\textit{0}}) Y_{lm}(\theta, \varphi),\] where

\begin{equation}\label{eq:fl}F_l(\eta) =
    \begin{cases}
      1, \ l =0,\\
      (\eta_\infty - \eta)^\nu\big(K_1^{(l)}J_\nu(z_l(\eta_\infty - \eta))+K_2^{(l)}Y_{\nu}(z_l(\eta_\infty - \eta))\big), \ l\in\mathbb{N},
    \end{cases}\
\end{equation}

\[ K_1^{(l)} = \frac{\pi z_l  Y_{\nu-1}(z_l\eta_\infty)}{2 \eta_\infty^{\nu-1}}, \ \ \ K_2^{(l)} = -\frac{\pi z_l  J_{\nu-1}(z_l\eta_\infty)}{2 \eta_\infty^{\nu-1}},\] $z_l = \frac{c\sqrt{l(l+1)}}{r},$  $\nu = \frac{c^2\eta_\infty}{2D}+1$ and $\textbf{\textit{0}}$ denotes a point on the unit sphere with the spherical coordinates $\theta = 0, \ \varphi=0.$

\end{theorem}

\begin{proof}

Let $\widetilde{u} = L(\theta)Z(\varphi)F(\eta)$ be the solution of \eqref{eq_main}, then, after multiplying the equation by $\frac{r^2sin^2\theta}{\widetilde{u}},$ one gets

\[ r^2 \sin^2\theta E(\eta)  \frac{F'}{F} + \frac{r^2 \sin^2\theta}{c^2}\frac{F''}{F} - \sin^2\theta \frac{L''}{L}  - \sin \theta \cos \theta \frac{L'}{L}  = \frac{Z''}{Z} = -m^2,\] where $E(\eta) = \frac{e(\eta)}{D} + \frac{e'}{c^2e},$ and $m$ is a separation constant. Thus, $Z(\varphi)=e^{im\varphi}$ and as it is $2\pi$-periodic, $m$ must be an integer.

Now separating the independent variable $\theta$ one gets
\[r^2 E(\eta) \frac{F'}{F} + \frac{r^2}{c^2}\frac{F''}{F}= \frac{L''}{L} + cot\theta \frac{L'}{L} - \frac{m^2}{\sin^2\theta} = -l(l+1).\] By substituting  $x = cos\theta,$ this equation for $\theta$ is equivalent to the next associated Legendre equation \cite[page 648]{arfken1972mathematical}

\[ (1-x^2)\frac{d^2L}{dx^2} -2x \frac{dL}{dx} -\frac{m^2 L}{1-x^2} +l(l+1)L = 0.\] The solution of the above equation has a singularity if $l\notin\mathbb{Z}.$ If $l\in\mathbb{Z},$ then the general solution is $C_1 P_l^m(cos \theta) + C_2 Q_l^m(cos \theta),$ where $P_l^m(\cdot)$ and $Q_l^m(\cdot)$ are the associated Legendre polynomial and the Legendre function of the second kind. As $Q_l^m(x)$ is singular at $x = \pm 1,$ $L(\theta) = P_l^m(\cos\theta).$

As $e(\eta) = \frac{\eta_\infty}{\eta_\infty-\eta},$ the equation for the separated independent variable $\eta$ is
\begin{equation}
\label{eq_eta}
F'' + \frac{\eta_\infty c^2+D}{D(\eta_\infty - \eta)}F' + \frac{l(l+1)c^2}{r^2} F = 0.
\end{equation}

Let us first solve the above equation for $l = 0$ and denote the solution by $F_0(\eta).$ By denoting $F_0' = P,$ the equation \eqref{eq_eta} is equivalent to

\[P' + \frac{\eta_\infty c^2+D}{D(\eta_\infty - \eta)}P = 0,\] from which follows that $P = K_1^{(0)}(\eta_\infty-\eta)^{\frac{\eta_\infty c^2}{D}+1},$ and $F_0 = \int P d\eta = K_1^{(0)}(\eta_\infty-\eta)^{2\nu} + K_2^{(0)},\ \nu = \frac{\eta_\infty c^2}{2D}+1,$ where the superscript $(0)$ means that the constants $K_1^{(0)}$ and $K_2^{(0)}$ correspond to the solution of \eqref{eq_eta} with $l = 0$.

Now let us solve the equation \eqref{eq_eta} for $l\in \mathbb{N}.$  By setting  $x=\eta_\infty-\eta,$ one transforms \eqref{eq_eta} to

\[ \frac{d^2F}{dx^2} + \left(-\left(\frac{\eta_\infty c^2}{D}+2\right) + 1\right)\frac{1}{x}\frac{dF}{dx}  + \frac{l(l+1)c^2}{r^2} F = 0.\] Denote the general solution of the above equation by $F_l$ for $l\in\mathbb{N}.$ Then, by \cite[page 97]{watson1944} one gets

\[ F_l = (\eta_\infty - \eta)^\nu\left(K_1^{(l)} J_\nu(z_l(\eta_\infty - \eta)) + K_2^{(l)} Y_{\nu}(z_l(\eta_\infty - \eta))\right).\]

As $Y_{lm}(\theta, \varphi) = d_{lm}P_l^m(cos\theta)e^{im\phi}, \ d_{lm} = (-1)^m\sqrt{\frac{2l+1}{4\pi}\frac{(l-m)!}{(l+m)!}},$ then, the general solution of \eqref{eq_main} is

\begin{equation}\label{fund_sol} \sum_{l=0}^\infty F_l(\eta) \sum_{m = -l}^l d_{lm}^{-1}Y_{lm}(\theta, \varphi),\end{equation} where

\begin{equation}\label{func_f}F_l(\eta) =
    \begin{cases}
      K_1^{(0)}(\eta_\infty - \eta)^{2\nu} + K_2^{(0)}, \ l =0,\\
      (\eta_\infty - \eta)^\nu\big(K_1^{(l)}J_\nu(z_l(\eta_\infty - \eta))+K_2^{(l)}Y_{\nu}(z_l(\eta_\infty - \eta))\big), \ l\in\mathbb{N}.
    \end{cases}\
\end{equation} Due to the spherical harmonic closure relations  \cite[1.17.25]{DLMF}, the initial conditions \eqref{cond} are equivalent to the following system of conditions

\begin{equation}
    \label{condit}
    \begin{cases}
     d_{lm}^{-1} F_l\big|_{\eta=0} = Y_{lm}^*(\textbf{\textit{0}}),\\
     F_l'\big|_{\eta=0} = 0,
    \end{cases}\
\end{equation} for all $l \in 0 \cup \mathbb{N}.$

Let us consider the case $l = 0.$ The above conditions become

\begin{equation*}
    \begin{cases}
     \big(K_1^{(0)}(\eta_\infty - \eta)^{2\nu} + K_2^{(0)}\big)\big|_{\eta=0} = d_{00}Y_{00}^*(\textbf{\textit{0}}),\\
     K_1^{(0)}2\nu(\eta_\infty - \eta)^{2\nu-1}\big|_{\eta=0} = 0.
    \end{cases}\
\end{equation*} One can see that $K_1^{(0)} = 0$ and $K_2^{(0)} = d_{00}Y_{00}^*(\textbf{\textit{0}}).$

For $l\in\mathbb{N}$ the conditions \eqref{condit} become

\begin{equation*}
    \begin{cases}
     (\eta_\infty - \eta)^{\nu}\big(K_1^{(l)} J_\nu(z_l(\eta_\infty - \eta)) + K_2^{(l)}Y_{\nu}(z_l(\eta_\infty - \eta))\big)\big|_{\eta=0} = d_{lm}Y_{lm}^*(\textbf{\textit{0}}),\\
     (\eta_\infty - \eta)^{\nu}\big( K_1^{(l)} J_\nu(z_l(\eta_\infty - \eta)) + K_2^{(l)}Y_{\nu}(z_l(\eta_\infty - \eta))\big) \big)'\big|_{\eta=0} = 0.
    \end{cases}\
\end{equation*} The last system is equivalent to

\begin{equation*}
    \begin{cases}
     A_1^{(l)}K_1^{(l)} + B_1^{(l)}K_2^{(l)}  = d_{lm}Y_{lm}^*(\textbf{\textit{0}}), \\
     A_2^{(l)}K_1^{(l)} + B_2^{(l)}K_2^{(l)}  = 0
    \end{cases}\
\end{equation*}where
\[ A_1^{(l)} = \eta_\infty^\nu J_\nu(z_l\eta_\infty), \ \ \ A_2^{(l)} =  - \nu \eta_\infty^{\nu-1}J_\nu(z_l\eta_\infty) - \frac{z_l}{2}\eta_\infty^\nu(J_{\nu-1}(z_l\eta_\infty) - J_{\nu+1}(z_l\eta_\infty)),\]
\[ B_1^{(l)} = \eta_\infty^\nu Y_{\nu}(z_l\eta_\infty), \ \ \  B_2^{(l)} =  - \nu \eta_\infty^{\nu-1}Y_{\nu}(z_l\eta_\infty) - \frac{z_l}{2}\eta_\infty^\nu(Y_{\nu-1}(z_l\eta_\infty) - Y_{\nu+1}(z_l\eta_\infty)).\]  Thus,

\[ K_1^{(l)} = \frac{B_2^{(l)}d_{lm}Y_{lm}^*(\textbf{\textit{0}})}{A_1^{(l)}B_2^{(l)}-A_2^{(l)}B_1^{(l)}}, \ \  K_2^{(l)} = - \frac{A_2^{(l)}d_{lm}Y_{lm}^*(\textbf{\textit{0}})}{A_1^{(l)}B_2^{(l)}-A_2^{(l)}B_1^{(l)}}.\] After straightforward algebraic manipulations, one can see that
\[ A_1^{(l)}B_2^{(l)}-A_2^{(l)}B_1^{(l)} = \frac{z_l}{2}\eta_\infty^{2\nu} \bigg( Y_{\nu}(z_l\eta_\infty) \big( J_{\nu-1}(z_l\eta_\infty) - J_{\nu+1}(z_l\eta_\infty)\big) \] \[- J_{\nu}(z_l\eta_\infty)\big(Y_{\nu-1}(z_l\eta_\infty) - Y_{\nu+1}(z_l\eta_\infty)\big) \bigg) \]
\[ = \frac{z_l}{2}\eta_\infty^{2\nu}\bigg( -\big(J_{\nu+1}(z_l\eta_\infty)Y_{\nu}(z_l\eta_\infty) - J_{\nu}(z_l\eta_\infty)Y_{\nu+1}(z_l\eta_\infty)\big)\] \[ - \big(J_{\nu}(z_l\eta_\infty)Y_{\nu-1}(z_l\eta_\infty) - J_{\nu-1}(z_l\eta_\infty)Y_{\nu}(z_l\eta_\infty) \big)\bigg).\] Using the Wronskian expression $W(J_\nu(x), Y_{\nu}(x)) = J_{\nu+1}(x)Y_{\nu}(x) - J_{\nu}(x)Y_{\nu+1}(x) = \frac{2}{\pi x}$ (see \cite[10.5]{DLMF}), one gets
\[ A_1^{(l)}B_2^{(l)}-A_2^{(l)}B_1^{(l)} = -\frac{z_l}{2} \eta_\infty^{2\nu} \frac{4}{\pi \eta_\infty z_l} =- \frac{2\eta_\infty^{2\nu-1}}{\pi}.\] From which it follows that

\[ K_1^{(l)} = \frac{\pi d_{lm} \bigg(\nu Y_{\nu}(z_l\eta_\infty)+\frac{z_l\eta_\infty}{2}(Y_{\nu-1}(z_l\eta_\infty)-Y_{\nu+1}(z_l\eta_\infty))\bigg)Y_{lm}^*(\textbf{\textit{0}})}{2 \eta_\infty^{\nu}}, \] \[ K_2^{(l)} = -\frac{\pi d_{lm} \bigg(\nu J_\nu(z_l\eta_\infty)+\frac{z_l\eta_\infty}{2}(J_{\nu-1}(z_l\eta_\infty)-J_{\nu+1}(z_l\eta_\infty))\bigg)Y_{lm}^*(\textbf{\textit{0}})}{2 \eta_\infty^{\nu} }.\] By subsequently applying the identities $Y_{\nu}'(x) = \frac{1}{2}(Y_{\nu-1}(x) - Y_{\nu+1}(x))$ and $xY_\nu'(x) = xY_{\nu-1}(x)-\nu Y_{\nu}(x)$, one gets $\nu Y_\nu(x) + \frac{x}{2}(Y_{\nu-1}(x)-Y_{\nu+1}(x)) = x Y_{\nu-1}(x).$ Thus,

\[ K_1^{(l)} = \frac{\pi d_{lm} z_l  Y_{\nu-1}(z_l\eta_\infty)Y_{lm}^*(\textbf{\textit{0}})}{2 \eta_\infty^{\nu-1}}, \] Analogous transformations for the functions $J_\nu(\cdot)$  lead to

\[ K_2^{(l)} = -\frac{\pi d_{lm} z_l  J_{\nu-1}(z_l\eta_\infty)Y_{lm}^*(\textbf{\textit{0}})}{2 \eta_\infty^{\nu-1}}.\] By putting the above constants into \eqref{func_f} and \eqref{fund_sol}, one finishes the proof of the theorem.\end{proof}

The following results derive some properties of the functions $F_l(\eta),$ which will be used later.

\begin{lemma}
\label{lemma1} For a fixed $\eta\in[0,\eta_\infty)$ and any constant $a>1$ the following asymptotic behaviour holds true
\begin{equation}\label{lemma1eq}Y_{a-1}(z_l\eta_\infty) J_{a}(z_l(\eta_\infty-\eta)) -  J_{a-1}(z_l\eta_\infty)Y_{a}(z_l(\eta_\infty-\eta)) = \frac{2\cos(z_l\eta)}{\pi z_l \sqrt{\eta_\infty(\eta_\infty-\eta)}} + O_\eta(z_l^{-2}), \end{equation} as $ l\to\infty,$ where the terms $O_\eta(z_l^{-2})$ may depend on $\eta.$

\end{lemma}

\begin{proof}
For the Bessel's function of the first kind $J_{a}(x)$  with $a\geq-\frac{1}{2}$ the following holds true uniformly in $x\geq0$

\[
  \left|J_a(x) - \left( \frac{2}{\pi x}  \right)^{\frac{1}{2}} \cos\left(x-\frac{1}{2}a \pi - \frac{1}{4}\pi\right)\right| \leq \frac{d_a}{x^{\frac{3}{2}}},
\]
 where $d_a$ is a constant depending on $a,$ see   \cite[Theorem 4.1]{olenko2006upper}.

An approximation for $Y_\nu(x)$ follows from the analogous estimates  and the relationship \cite[Section 3.6]{watson1944}
\[ Y_a(x) = \frac{H_a^{(1)}(x) - H_a^{(2)}(x)}{2i},\] where $H_a^{(1)}(x)$ and $H_a^{(2)}(x)$ are the Hankel functions of the first and second kind respectively. A straightforward modification of the proof of \cite[Theorem 4.1]{olenko2006upper} gives that, uniformly in $x\geq2,$  for the Bessel's function of the second kind $Y_{a}(x),$  $a\geq-{1}/{2},$ it holds true

\begin{equation}\label{eq:approx} \left|Y_a(x) - \left( \frac{2}{\pi x}  \right)^{\frac{1}{2}}\sin\left(x-\frac{1}{2}a \pi - \frac{1}{4}\pi\right)\right| \leq \frac{d_a}{x^{\frac{3}{2}}}.\end{equation}

Thus, left-hand side of \eqref{lemma1eq} is asymptotically equal to
\[ \frac{2}{\pi z_l \sqrt{\eta_\infty(\eta_\infty - \eta)}}\left( \sin\left(z_l \eta_\infty - \frac{(a-1)\pi}{2} - \frac{\pi}{4}\right)  \cos\left(z_l(\eta_\infty-\eta) - \frac{a\pi}{2} - \frac{\pi}{4}\right) \right.\] \[ \left.- \cos\left(z_l \eta_\infty - \frac{(a-1)\pi}{2} - \frac{\pi}{4}\right)\sin\left(z_l (\eta_\infty -\eta)- \frac{a\pi}{2} - \frac{\pi}{4}\right)\right)  + O_\eta(z_l^{-2})\]

\[ = \frac{2}{\pi z_l \sqrt{\eta_\infty(\eta_\infty - \eta)}}\left( \cos\left(z_l \eta_\infty - \frac{a\pi}{2} - \frac{\pi}{4}\right)  \cos\left(z_l(\eta_\infty-\eta) - \frac{a\pi}{2} - \frac{\pi}{4}\right) \right.\] \[ \left.+ \sin\left(z_l \eta_\infty - \frac{a\pi}{2} - \frac{\pi}{4}\right)\sin\left(z_l (\eta_\infty -\eta)- \frac{a\pi}{2} - \frac{\pi}{4}\right)\right)  + O_\eta(z_l^{-2})\]

\[ = \frac{2\cos(z_l\eta)}{\pi z_l \sqrt{\eta_\infty(\eta_\infty-\eta)}} + O_\eta(z_l^{-2}). \]

\end{proof}

\begin{lemma}\label{lemma2.1}
For a fixed $\eta\in[0,\eta_\infty)$ the following asymptotic behaviour holds true
\[ F_l(\eta) = \frac{\cos(z_l \eta)}{e^{\nu-1/2}(\eta)}  + O_\eta(z_l^{-1}), \ \ \ l\to\infty, \]
 where the terms $O_\eta(z_l^{-1})$ may depend on $\eta.$
\end{lemma}

\begin{proof}
Let us consider separately the following expression

\[ K_1^{(l)}J_\nu(z_l(\eta_\infty - \eta))+K_2^{(l)}Y_{\nu}(z_l(\eta_\infty - \eta))  \] \[= \frac{\pi z_l  }{2 \eta_\infty^{\nu-1}}(Y_{\nu-1}(z_l\eta_\infty)J_\nu(z_l(\eta_\infty - \eta)) - J_{\nu-1}(z_l\eta_\infty) Y_{\nu}(z_l(\eta_\infty - \eta))).\] According to Lemma~\ref{lemma1} the above expression equals to

\[ \frac{\sqrt{e(\eta)}\cos(z_l\eta)}{\eta_\infty^\nu} + O_\eta(z_l^{-1}).\] Thus,
\[ F_l(\eta) = \frac{(\eta_\infty - \eta)^\nu\sqrt{e(\eta)}\cos(z_l\eta)}{\eta_\infty^\nu} + O_\eta(z_l^{-1}) = \left(\frac{\eta_\infty - \eta}{\eta_\infty}\right)^\nu \cos(z_l \eta) \sqrt{e(\eta)} + O_\eta(z_l^{-1})   \] \[ =\frac{\cos(z_l \eta)}{e^{\nu-1/2}(\eta)} + O_\eta(z_l^{-1}).\]

\end{proof}

\begin{lemma}\label{lemma2}
The functions $F_l(\eta)$ are uniformly bounded on $\eta\in[0,\eta_\infty)$ and $\ l=0,1,...\ .$
\end{lemma}

\begin{proof}
Note that $\nu>1.$ First, let us consider the absolute value of the first summand in \eqref{eq:fl}
\[ |z_l(\eta_\infty-\eta)^\nu Y_{\nu-1}(z_l\eta_\infty) J_\nu(z_l(\eta_\infty-\eta))|.\] It is bounded by some constant $C$ for all values of $\eta\in[0,\eta_\infty)$ and $l.$ Indeed, $|(z_l(\eta_\infty-\eta))^\frac{1}{2} J_\nu(z_l(\eta_\infty-\eta))|<C,$ as $|\sqrt{x}J_\nu(x)|,  x\in[0,\infty),$ is a bounded function \cite{olenko2006upper}. The terms $|z_l^\frac{1}{2}Y_{\nu-1}(z_l\eta_\infty)|$ are uniformly bounded for all $l$ as the function $\sqrt{x} Y_\nu(x)$ is bounded on $[C,\infty]$ for any fixed $C>0.$ The later boundedness follows from the continuity of $\sqrt{x} Y_\nu(x)$ on $[C,\infty]$ and the inequality~\eqref{eq:approx}.

The absolute value of the second summand in \eqref{eq:fl} can be estimated as
\begin{equation}\label{eq_uni} |z_l  (\eta_\infty-\eta)^\nu J_{\nu-1}(z_l\eta_\infty)Y_\nu(z_l(\eta_\infty-\eta))|\leq  C |z_l^{\frac{1}{2}}(\eta_\infty-\eta)^\nu Y_\nu(z_l(\eta_\infty-\eta))|,
\end{equation} which is obtained by using the boundedness of the function $|\sqrt{x}J_\nu(x)|, \ x\in[0,\infty).$  The boundedness of the function on the right-hand side of  \eqref{eq_uni} follows from the asymptotic behaviour $Y_\nu(x)\sim\frac{C}{x^\nu},$ $x\to0,$ see \cite[9.1.11]{abramowitz1968handbook},  \eqref{eq:approx} and the continuity of the function $Y_\nu(x)$ on the interval $x\in[C,\infty).$ \end{proof}

\section{Solution for stochastic spherical diffusion equation}\label{sect_rand}

In this section we consider the case of equations from Section~\ref{sec3} with the initial conditions determined by an isotropic Gaussian random field.

Namely, we consider the following initial condition and its spherical harmonics expansion
\begin{equation} \label{cond_rand1}
{u}(\eta, \theta, \varphi)|_{\eta = 0} = T(\theta,\varphi) = \sum_{l=0}^{\infty}\sum_{m=-l}^la_{lm}Y_{lm}(\theta,\varphi),
\end{equation}
\begin{equation} \label{cond_rand2}
\frac{\partial {u}(\eta, \theta, \varphi)}{\partial \eta}{\biggl |}_{\eta = 0} =0,
\end{equation} where $a_{lm}, \ m=-l,...,l,\ l\geq0,$ are Gaussian random variables. Without lost of generality, we assume that the random field $T(\theta, \varphi)$ is centered $ET(\theta,\varphi) = 0.$

\begin{lemma}
\label{lemma_cont}
If the angular power spectrum $\{C_l, l=0,1,...\}$ of the Gaussian isotropic random field $T(\theta, \varphi)$ satisfies the condition

\begin{equation}\label{cond_spec}
\sum_{l=1}^\infty  C_l  l^{10}(2l+1) < +\infty,
\end{equation} then $T(\theta,\varphi)\in C^2(\mathbb{S}^2)$ a.s.
\end{lemma}

\begin{proof} It follows from (\ref{add_form}) that
\[Cov(T(\theta, \varphi), T(\theta', \varphi')) = \sum_{l=0}^\infty C_l\sum_{m=-l}^lY_{lm}(\theta, \varphi) Y^*_{lm}(\theta', \varphi')\]
\[ = (4\pi)^{-1} \sum_{l=0}^\infty   C_l  (2l+1)P_l(\cos \Theta) .\]  Therefore, the statement of the lemma directly follows from {\rm\cite[ Lemma 3.3]{cheng2016excursion}}.\end{proof}

Lemma~\ref{lemma_cont} provides the sufficient conditions for the random field $T(\theta,\varphi)$ to be a.s. twice continuously differentiable with respect to $\theta,\varphi.$ In the following results we assume that $T(\theta,\varphi)$ is a.s. twice continuously differentiable or that the conditions of Lemma~\ref{lemma_cont} hold true.

\begin{lemma}\label{lemma:const}
Let  $K_T(\Theta)$ be a covariance function of  an isotropic a.s. continuous Gaussian random field $T(\theta,\varphi).$ There exists $\Theta'>0$ such $K_T(\Theta)=K_T(0),$ for all $\Theta \leq \Theta',$ if and only if $T(\theta,\varphi)=\xi, \ \xi\in L^2(\Omega),$ for all $\theta$ and $\varphi,$ with probability 1.
\end{lemma}

\begin{proof}
As the sufficiency is straightforward, let us proceed with the necessity.

Note that for any two spherical points $(\theta,\varphi)$ and $(\theta',\varphi')$ at the angular distance $\Theta,$ it holds \[E(T(\theta,\varphi)- T(\theta',\varphi'))^2=2(K_T(0)-K_T(\Theta)) =0,\]
which implies that $T(\theta,\varphi) = T(\theta',\varphi')$ with probability 1. Due to the isotropy of the random field $T(\theta,\varphi)$ it is also true for any spherical points. \end{proof}

\begin{theorem}
\label{th_stoch}
The solution $u(\eta, \theta, \varphi)$ of the initial random value problem \eqref{eq_main}, \eqref{cond_rand1}, \eqref{cond_rand2} is given by the following random series
\begin{equation}
\label{sol_rand}
u(\eta,\theta,\varphi) = \sum_{l=0}^{\infty} F_l(\eta) \sum_{m=-l}^l  a_{lm}Y_{lm}(\theta, \varphi).
\end{equation} The covariance function of $u(\eta, \theta, \varphi)$ is given by

\begin{equation}\label{correl}
Cov(u(\eta, \theta, \varphi), u(\eta', \theta', \varphi')) =(4\pi)^{-1}\sum_{l=0}^\infty   C_l  (2l+1) F_l(\eta)F_l(\eta')P_l(\cos \Theta) ,\end{equation} where $\eta,\eta'\in[0,\eta_\infty),\ \theta,\theta'\in[0,\pi], \ \varphi,\varphi'\in[0,2\pi),$ $P_l(\cdot)$ is the $l$th Legendre polynomial, and $\Theta$ is the angular distance between the points $(\theta,\varphi)$ and $(\theta',\varphi').$

\begin{remark}
The convergence of the series in \eqref{sol_rand} is understood in the $L_2(\Omega\times \mathbb{S}^2)$  sense, that  is
\[ \lim\limits_{L\to\infty}E\left( \int\limits_{0}^{\pi}\int\limits_{0}^{2\pi} \left( u(\eta,\theta,\varphi) - \sum_{l=0}^{L} F_l(\eta) \sum_{m=-l}^l  a_{lm}Y_{lm}(\theta, \varphi)\right)^2 \sin\theta d\varphi d\theta \right) =0.\] However, the series in \eqref{sol_rand} also converges almost surely, see {\rm \cite[Section 2]{lang2015isotropic}}.
\end{remark}

\begin{proof}

The solution of the initial value problem \eqref{eq_main}, \eqref{cond_rand1}, \eqref{cond_rand2} is a spherical convolution of the function $\widetilde{u}(\cdot)$ obtained in Theorem~\ref{th1} and the random field $T(\theta,\varphi),$ provided that the corresponding Laplace series converges in the Hilbert space $L_2(\Omega\times \mathbb{S}^2).$

Let the two functions $f_1(\cdot)$ and $f_2(\cdot)$ on the sphere $\mathbb{S}^2$ belong to the space $L_2(\mathbb{S}^2)$ and have the Fourier-Laplace coefficients
\[ a^{(i)}_{lm} = \int_{\mathbb{S}^2}f_i(\theta,\varphi)Y_{lm}^*(\theta, \varphi) \sin\theta d\theta d\varphi, \ \ i=1,2. \] The non-commutative spherical convolution of $f_1(\cdot)$ and $f_2(\cdot)$ is defined as the Laplace series (see \cite{dunkel2009relativistic})
\begin{equation}\label{conv}
 [f_1*f_2](\theta, \varphi) = \sum\limits_{l=0}^\infty \sum\limits_{m=-l}^l a_{lm}^*Y_{lm}(\theta, \varphi)
\end{equation} with the Fourier-Laplace coefficients given by
\[ a_{lm} = \sqrt{\frac{4\pi}{2l+1}} a_{lm}^{(1)}a_{l0}^{(2)}, \] provided that the series in \eqref{conv} converges in the corresponding Hilbert space.

Thus, the random solution $u(\eta, \theta, \varphi)$ of equation  \eqref{eq_main} with the initial conditions~\eqref{cond_rand1} and~\eqref{cond_rand2} can be written as a spherical random field with the following Laplace series representation

\[ u(\eta, \theta, \varphi) = [T*u](\theta,\varphi) = \sum_{l=0}^{\infty}  F_l(\eta) \sum_{m=-l}^l  \sqrt{\frac{4\pi}{2l+1}} a_{lm} Y_{l0}^*(\textbf{\textit{0}}) Y_{lm}(\theta, \varphi) \]
\[= \sum_{l=0}^{\infty} F_l(\eta) \sum_{m=-l}^l  a_{lm}Y_{lm}(\theta, \varphi).\]
In the above series, the identity $Y_{l0}^*(\textbf{\textit{0}}) = \sqrt{\frac{2l+1}{4\pi}}$ was applied to simplify the expression.  By  Lemma~\ref{lemma2} and condition \eqref{cond_spec} the above series converges in $L_2(\Omega\times\mathbb{S}^2).$

By applying the addition formula (\ref{add_form}) for the spherical harmonics, one obtains the covariance function of $u(\eta,\theta,\varphi)$ in the form
\[Cov(u(\eta, \theta, \varphi), u(\eta', \theta', \varphi')) = \sum_{l=0}^\infty C_l F_l(\eta)F_l(\eta')\sum_{m=-l}^lY_{lm}(\theta, \varphi) Y^*_{lm}(\theta', \varphi')\]
\[ = (4\pi)^{-1} \sum_{l=0}^\infty   C_l  (2l+1) F_l(\eta) F_l(\eta')P_l(\cos \Theta) .\] As $|P_l(\cos(\Theta))|\leq1,$ as $l\to\infty,$ it follows from \eqref{eq:ang_conv} and Lemma~\ref{lemma2} that the above series is convergent.
\end{proof}

\end{theorem}

\section{Properties of stochastic solutions and their approximations}
\label{sec:approx}
This section investigates time-smoothness properties of the solution and truncation errors of its approximation.

For $L\in \mathbb{N},$ the following truncated series are used to approximate the solution of the random initial value problem in Theorem~\ref{th_stoch}

\begin{equation}
\label{approx}
u_L(\eta, \theta, \varphi) = \sum_{l=0}^{L} F_l(\eta) \sum_{m=-l}^l  a_{lm}Y_{lm}(\theta, \varphi).
\end{equation}

The next result provides the upper bounds for the corresponding approximation error.

\begin{theorem}
\label{approx_th}

Let $u(\eta, \theta, \varphi)$ be the solution of the initial value problem \eqref{eq_main}, \eqref{cond_rand1}, \eqref{cond_rand2} and $u_L(\eta, \theta, \varphi)$ be its approximation. Then,  for $\eta\in[0,\eta_\infty)$ the following bound for the truncation error holds true
\[ \left|\left| u(\eta, \theta, \varphi) - u_L(\eta, \theta, \varphi)\right|\right|_{L_2(\Omega \times \mathbb{S}^2)} \leq C \left(\sum_{l=L+1}^{\infty} C_l (2l+1) \right)^{1/2}, \] where the constant $C$ does not depend on $\eta.$  \end{theorem}   \begin{proof}

The truncation error field $\widehat{u}_L(\eta, \theta, \varphi) = u(\eta, \theta, \varphi) - u_L(\eta, \theta, \varphi), \ L \in \mathbb{N},$ is a centered Gaussian random field, i.e. $E\widehat{u}_L(\eta, \theta, \varphi) = 0$ for all $L\in\mathbb{N},\ \theta\in[0,\pi], \ \varphi\in[0,2\pi),$ and $\eta\in[0,\eta_\infty).$ It follows from \eqref{sol_rand}, \eqref{approx} and the orthogonality of $\{a_{lm}\}$ that

\[  \left|\left| u(\eta, \theta, \varphi) - u_L(\eta, \theta, \varphi)\right|\right|^2_{L_2(\Omega \times S^2)} =  \int\limits_{0}^{\pi}\int\limits_{0}^{2\pi} E\bigg( \sum\limits_{l=L+1}^\infty F_l(\eta)\sum\limits_{m=-l}^l a_{lm} Y_{lm}(\theta, \varphi) \bigg)^2 \sin\theta  d\varphi d\theta.\] Then, due to the addition theorem for the spherical harmonics
\[||u(\eta,\theta,\varphi)-u_L(\eta,\theta,\varphi)||^2_{L_2(\Omega\times\mathbb{S}^2)}=C\sum_{l=L+1}^{\infty} C_l (2l+1) F_l^2(\eta).\] The statement of the theorem follows from Lemma~\ref{lemma2}.

\end{proof}

\begin{theorem}Let $u(\eta, \theta, \varphi)$ be the solution of the initial value problem \eqref{eq_main}, \eqref{cond_rand1}, \eqref{cond_rand2}. If

\[ \label{assumth4}\sum\limits_{l=1}^\infty C_ll^3 <+\infty,\] then for each $\eta\in[0,\eta_\infty)$ and $h\in (0,\eta_\infty-\eta)$

\[ ||u(\eta+h,\theta,\varphi)-u(\eta,\theta,\varphi)||_{L_2(\mathbb{S}^2\times \Omega)} \leq Ch, \] where the constant $C$ does not depend on $\eta.$

\end{theorem}

\begin{proof}
For $h<\eta_\infty-\eta,$ let us consider
\[ ||u(\eta+h,\theta,\varphi)-u(\eta,\theta,\varphi)||^2_{L_2(\Omega\times \mathbb{S}^2)} =  \sum\limits_{l=1}^\infty C_l(2l+1) (F_l(\eta+h)-F_l(\eta))^2.\] By \eqref{eq:fl} and applying the Cauchy inequality, one obtains that
\[ ||u(\eta+h,\theta,\varphi)-u(\eta,\theta,\varphi)||^2_{L_2(\Omega\times\mathbb{S}^2)} \leq C \sum\limits_{l=1}^\infty C_l (2l+1) (A^2_l(\eta,h)+B^2_l(\eta,h)),\] where
\[A_l(\eta,h):= K_1^{(l)}\big( \eta_\infty - (\eta+h))^\nu J_\nu(z_l(\eta_\infty - (\eta+h))) - (\eta_\infty - \eta)^\nu J_\nu(z_l(\eta_\infty - \eta)) \big)\] and
\[B_l(\eta,h):= K_2^{(l)}\big( \eta_\infty - (\eta+h))^\nu Y_{\nu}(z_l(\eta_\infty - (\eta+h))) - (\eta_\infty - \eta)^\nu Y_{\nu}(z_l(\eta_\infty - \eta)) \big).\]
First, let us estimate $|A_l(\eta,h)|.$ Let $f(x):=x^\nu J_{\nu}(x),$ then
\[ A_l(\eta,h) = K_1^{(l)}{z_l^{-\nu}}(f(z_l(\eta_\infty - (\eta+h)))-f(z_l(\eta_\infty - \eta))).\]
  By the mean value theorem
\[|f(z_l(\eta_\infty - (\eta+h)))-f(z_l(\eta_\infty - \eta))| \leq z_l h \max_x|(x^\nu J_\nu(x))'|,\] where the maximum is taken over the interval $[z_l(\eta_\infty - (\eta+h)), z_l(\eta_\infty - \eta)].$ As $(x^\nu J_\nu(x))' = x^\nu J_{\nu-1}(x),$ one obtains
\[{z_l^{-\nu}} |f(z_l(\eta_\infty - (\eta+h)))-f(z_l(\eta_\infty - \eta))|\leq   {h}z_l^{-\nu+1} \]
\[\times \max_{\alpha\in[0,h]} |(z_l( \eta_\infty - (\eta+\alpha)))^\nu J_{\nu-1}(z_l(\eta_\infty - (\eta+\alpha)))|.\] As $\sqrt{x}J_\nu(x)$ is a bounded function on $[0,\infty)$ and $\nu>1,$ the next upper bound holds true
\[  |A_l(\eta,h)| \leq |K_1^{(l)}| C {h}z_l^{-\nu+1} \max_{\alpha\in[0,h]} (z_l( \eta_\infty - (\eta+\alpha)))^{\nu-\frac{1}{2}} \leq C|K_1^{(l)}| h \sqrt{z_l}.\]

For $B_l(\eta, h),$ by setting $g(\eta):= x^\nu Y_{\nu}(x)$ and using analogous calculations one obtains that
\[B_l(\eta,h)\leq {z_l^{-\nu}}|K_2^{(l)}|\cdot |g(z_l(\eta_\infty - (\eta+h)))-g(z_l(\eta_\infty - \eta))| \leq   {h}{z_l^{-\nu+1}}|K_2^{(l)}|\]\[\times\max_{\alpha\in[0,h]}| (z_l( \eta_\infty - (\eta+\alpha)))^\nu Y_{\nu-1}(z_l(\eta_\infty - (\eta+\alpha)))| \leq  {h}{z_l^{-\nu+1}}|K_2^{(l)}|\max_{x\in[0,z_l\eta_\infty]}|x^\nu Y_{\nu-1}(x)|\]\[\leq C{h}{z_l^{-\nu+1}}|K_2^{(l)}|(z_l\eta_\infty)^{\nu-\frac{1}{2}}\leq C|K_2^{(l)}|hz_l^{\frac{1}{2}},\] as $|x^\nu Y_{\nu-1}(x)|$ is bounded on $[0,A], \ A>0,$ and $|x^\nu Y_{\nu-1}(x)|\leq Cx^{\nu-\frac{1}{2}}$ when $x\geq A.$

Using the asymptotics of the functions $J_\nu(\cdot)$ and $Y_\nu(\cdot),$ it is straightforward to verify that the constants $K_1^{(l)}$ and $K_1^{(2)}$ are bounded as $|K_1^{(l)}|\leq C\sqrt{z_l}$ and $|K_2^{(l)}|\leq C\sqrt{z_l}$ for all $l$. Thus, $A^2_l(\eta,h) \leq Ch^2z^2_l\sim Ch^2l^2.$ Analogously, $B^2_l(\eta,h)\leq Ch^2 z^2_l \sim Ch^2l^2$  which, by the assumption~(\ref{assumth4}), implies the statement of the theorem. \end{proof}

\section{On excursion probabilities}\label{sec_excur}

This section studies the excursion probabilities for the random solution obtained in section~\ref{sect_rand}. The main approach used in this section is based on the metric entropy theory. This section also studies errors of approximations of extremes of $u(\eta, \theta, \varphi)$ by extremes of the truncated field $u_L(\eta, \theta, \varphi).$ The approximation $u_L(\eta, \theta, \varphi)$ converges to $u(\eta, \theta, \varphi)$ in the space $L_2(\Omega \times \mathbb{S}^2),$ as $L\to\infty.$ However, in general, this type of convergence does not imply that the extremes of $u(\eta, \theta, \varphi)$ can be effectively approximated using the extremes of $u_L(\eta, \theta, \varphi).$ This section  obtains estimates of probabilities of large deviations between the extremes of $u(\eta, \theta, \varphi)$ and $u_L(\eta, \theta, \varphi).$

We first provide some results that will be used later to study properties of the distributions of extremes. In what follows $T$ is a subset of $\mathbb{R}^n.$

\begin{theorem} [{\rm \cite[Theorem 2.1.1.]{adler2007random}}]
\label{extreme} Let $\xi(x), \ x\in T,$ be a centered a.s. bounded Gaussian field, then  for $y\geq 0$
\begin{equation*}
P\left(  \sup\limits_{x\in T} \xi(x) - E\sup\limits_{x\in T} \xi(x) \geq y\right) \leq \exp\left({-\frac{y^2}{2\sigma_T^2}}\right),
\end{equation*} where $\sigma_T^2:=\sup\limits_{x\in T}E\xi^2(x),$ or, equivalently,
\begin{equation*} \label{sup}
P\left(  \sup\limits_{x\in T} \xi(x) \geq y \right) \leq \exp\left({-\frac{\left(y - E(\sup_{x\in T} \xi(x))\right)^2}{2\sigma_T^2}}\right)
\end{equation*} for $y\geq  E\sup\limits_{x\in T} \xi(x).$

\end{theorem}

Let us recall some results from the metric entropy theory. The canonical metric generated by $\xi(x), \ x\in T,$  is
\[ d(s,t) = \left( E(\xi(s) - \xi(t))^2 \right)^{\frac{1}{2}}, \  \ \ s,t\in T.\] Let $B(t,\varepsilon):=\{ s\in T: d(t,s)\leq \varepsilon\}$ be a ball in $T$ with the center at $t\in T$ and the radius $\varepsilon.$ Then, the following Fernique's inequality holds true \cite[Theorem 2.5.]{talagrand1996majorizing}
\[ E\sup\limits_{x\in T} \xi(x) \leq K\inf\limits_\mu\sup\limits_{x\in T} \int_{0}^{{\rm diam}(T)}\sqrt{\log\left(\frac{1}{\mu(B(x,\varepsilon))}\right)}d\varepsilon,\] where $\mu$ is a probability measure on $T,$ ${\rm diam}(T)$ is the diameter of $T$ in the metric $d(\cdot)$, and $K$ is a universal constant \cite{talagrand2014upper}. We will use this inequality to estimate $E\sup\limits_{\theta,\varphi}u(\eta, \theta, \varphi).$ In the following results $\eta\in[0,\eta_\infty)$  is fixed.

By \eqref{correl}, the canonical pseudometric generated by $u(\eta,\theta,\varphi)$ on the sphere is
\[ d_\eta(\Theta) := \left( E(u(\eta,\theta,\varphi) - u(\eta,\theta',\varphi'))^2 \right)^{\frac{1}{2}}=\sqrt{2}\left( Eu^2(\eta,\theta,\varphi) - cov(u(\eta,\theta,\varphi), u(\eta,\theta',\varphi')) \right)^{\frac{1}{2}}\]
\begin{equation}\label{eq:met} = (\sqrt{2\pi})^{-1}\left( \sum_{l=0}^\infty C_l(2l+1) F_l^2(\eta)(1-P_l(\cos(\Theta))) \right)^\frac{1}{2},\end{equation} where $\Theta$ is the angular distance between the points $(\theta, \varphi)$ and $(\theta', \varphi').$

Let us define "balls on the sphere" in terms of the above pseudometric $d_\eta(\cdot)$
\[ B_\eta((\theta,\varphi), \varepsilon) = \{ (\theta', \varphi'): d_\eta(\Theta)\leq \varepsilon \}.\]  Let us also introduce the function \begin{equation}\label{geps}
    g_\eta(\varepsilon) = \inf\{\Theta: \ d_\eta(\Theta) \geq \varepsilon\}.
\end{equation}
For the simplicity of the exposition, let us denote the covariance function of $u(\eta,\theta,\varphi)$ by
\[ K_\eta (\Theta) := Cov(u(\eta, \theta, \varphi), u(\eta, \theta', \varphi')) = (4\pi)^{-1}\sum_{l=0}^\infty   C_l  (2l+1) F^2_l(\eta)P_l(\cos \Theta),\] where $\Theta$ is the angular distance between the points $(\theta,\varphi)$ and $(\theta',\varphi').$ Note, that by \eqref{eq:met}  $K_\eta(\Theta)$ and $d_\eta(\Theta)$ are linked by the formula \[K_\eta(0) - K_\eta(\Theta)  = \frac{1}{2}d_\eta^2(\Theta). \]
\begin{theorem}\label{lm:finit} Let $u(\eta, \theta, \varphi)$ be the solution of the initial value problem \eqref{eq_main}, \eqref{cond_rand1}, \eqref{cond_rand2}. If
\begin{equation}\label{th6:cond} \sum\limits_{l=0}^\infty C_ll^{1+\beta} <+ \infty\end{equation} with $\beta \in (0,2],$ then
    \begin{equation}\label{estim}
        E\sup\limits_{\theta,\varphi}u(\eta, \theta, \varphi) \leq K \int\limits_0^R \sqrt{\log \left( \frac{2}{1-\cos(g_\eta(\varepsilon))}\right)}d\varepsilon,
    \end{equation}
    where $R=\max_{\Theta\in[0,\pi]} d_\eta(\Theta),$ and the integral in \eqref{estim} is finite.

\end{theorem}

\begin{proof}

By applying Fernique's inequality, it follows that
\[ E\sup\limits_{\theta,\varphi}u(\eta, \theta, \varphi) \leq K\inf\limits_\mu\sup\limits_{\theta, \varphi} \int_{0}^R\sqrt{\log\left( \frac{1}{\mu(B_\eta((\theta, \varphi),\varepsilon))}\right)}d\varepsilon.\] The selection of $\mu$ as the normalised uniform distribution on the sphere results in the upper bound
\[  E\sup\limits_{\theta,\varphi}u(\eta, \theta, \varphi) \leq K\sup\limits_{\theta, \varphi} \int_{0}^R\sqrt{\log\left(\frac{Leb(\mathbb{S}^2)}{Leb(B_\eta((\theta, \varphi),\varepsilon))}\right)}d\varepsilon.\]  As $u(\eta,\theta,\varphi)$ is an isotropic field, $Leb(B_\eta((\theta, \varphi),\varepsilon))$ is the same for all $\theta, \varphi$ and can be replaced with $Leb(B_\eta(\textbf{\textit{0}},\varepsilon))$. Note that $B_\eta(\textbf{\textit{0}},\varepsilon)$ is not necessarily a spherical cap as $d_\eta(\Theta)$ is not necessarily strictly increasing in $\Theta$. However, $B_\eta(\textbf{\textit{0}},\varepsilon)$ always contains a non-degenerate (not reduced to a single point) spherical cap. Indeed,  $u(\eta,\theta,\varphi)$ is $L_2$-continuous with respect to $\theta$ and $\phi$ due to condition \eqref{th6:cond}, see \cite[Lemma 4.3]{lang2015isotropic}. Then, its decomposition~\eqref{sol_rand} into spherical harmonics has the coefficients $a_{lm}F_l(\eta), \ l=1,2,...,$ that are non-zero with probability~1, see Lemma~\ref{lemma2}.  Hence, $u(\eta,\theta,\varphi)$ is not a constant random variable over a set of $\theta$ and $\phi$ and it follows from Lemma~\ref{lemma:const}  that the spherical cap is non-degenerate.

The value of  $ g_\eta(\varepsilon) \in [0,\pi]$  is the polar angle of the largest such cap.
 From the comparison of the corresponding areas it follows that $Leb(B_\eta(\textbf{\textit{0}},\varepsilon)) \geq 2\pi(1-\cos(g_\eta(\varepsilon))$ and
\[
 E\sup\limits_{\theta,\varphi}u(\eta, \theta, \varphi) \leq K \int\limits_0^R \sqrt{\log \left( \frac{2}{1-\cos(g_\eta(\varepsilon))}\right)}d\varepsilon.
\]

Now let us provide mild conditions that guarantee the convergence of the above integral  for any fixed $\eta\in[0,\eta_\infty).$

The Gaussian random field $u(\eta,\theta,\varphi)$ is a.s.~continuous and  isotropic and the corresponding spherical cap is non-degenerated. Due to Lemma \ref{lemma:const}, the continuous covariance function $K_\eta (\Theta)$ is non-constant in any neighbourhood of the origin. Thus, by the definitions \eqref{eq:met} and \eqref{geps}  $g_\eta(\varepsilon) \to 0,$ when $\varepsilon\to 0.$

  From $1-\cos(x) \sim x^2/2,$ $x\to 0,$ it follows that
  \[ \lim\limits_{x\to 0}\frac{\log\left(\frac{2}{1-\cos(x)}\right)}{-\log(x)} =
  \lim\limits_{x\to 0}\frac{\log\left(\frac{4}{x^2}\right)}{-\log(x)}=2\]
  and
 \[\log\left(\frac{2}{1-\cos(g_\eta(\varepsilon))}\right)\sim-2\log(g_\eta(\varepsilon)), \ \ \ \varepsilon\to0,\]
 which means that the integral in $\eqref{estim}$ is finite if and only if
\[ \int\limits_0^R\sqrt{-\log(g_\eta(\varepsilon))}d\varepsilon < \infty.\]

The above holds true if there is $\alpha>-1$ and $\varepsilon_0>0,$ such that for all $\varepsilon\leq\varepsilon_0$   it holds
$-\log (g_\eta(\varepsilon))  < C\varepsilon^{2\alpha},$ which gives $g_\eta(\varepsilon) > \exp({-C\varepsilon^{2\alpha}}).$
Then, it follows from the definition~(\ref{geps}) of $g_\eta(\varepsilon)$ that  $d_\eta(\exp({-C\varepsilon^{2\alpha}})) < \varepsilon.$ Denoting $\Theta = \exp({-C\varepsilon^{2\alpha}})$ we get that for $\beta = -{1}/{\alpha}\in({1}, \infty)$  and all small enough $\Theta$  it holds
\[ d_\eta(\Theta) < \frac{C}{(-\log \Theta)^{\beta/2}}.\] and

\[ K_\eta(0) -  K_\eta(\Theta)  = \frac{1}{2}d_\eta^2(\Theta)< \frac{C}{(-\ln \Theta)^{\beta}}. \]
 As $\log(1/\Theta) \leq  1/\Theta$   for sufficiently small $ \Theta,$ the later inequality is satisfied if
 \begin{equation}\label{cond1}
     K_\eta(0) - K_\eta(\Theta)  \leq  C \Theta^\beta.
     \end{equation}  Note, that by \cite[Lemma 4.2]{lang2015isotropic} the inequality (\ref{cond1})  holds true if
$ \sum\limits_{l=0}^\infty C_ll^{1+\beta} < \infty$ and $\beta \leq 2,$ which finishes the proof of  the theorem.
\end{proof}

Now we are ready to state the main results of this section.

\begin{theorem}\label{th:excur}
For each fixed $\eta\in[0,\eta_\infty)$ the following estimate holds true
\[ P\left(  \sup\limits_{\theta,\varphi}u(\eta,\theta,\varphi) > x \right) \leq \exp\left(-\frac{\left(x-E(\sup_{\theta,\varphi}u(\eta,\theta,\varphi))\right)^2}{2\sigma_\eta^2}\right),\] where
$\sigma_\eta^2 = (4\pi)^{-1}\sum_{l=0}^\infty   C_l  (2l+1) F^2_l(\eta).$

If
\begin{equation*} \sum\limits_{l=0}^\infty C_ll^{1+\beta} <+ \infty\end{equation*} with $\beta \in (0,2],$ then

\[ P\left(  \sup\limits_{\theta,\varphi}u(\eta,\theta,\varphi) > x \right) \leq \exp\left( -\frac{\left(x- K_1\right)^2}{2\sigma_{\eta,L}^2}
    \right)\] for $x> K_1:=K \int\limits_0^R \sqrt{\log \left( \frac{2}{1-\cos(g_{\eta}(\varepsilon))}\right)}d\varepsilon,$ where the latter integral is finite.
\end{theorem}
\begin{proof}
    The statement is a direct corollary of Theorems~\ref{extreme} and~\ref{lm:finit} as the application of the estimate for $E\sup\limits_{\theta,\varphi}u(\eta,\theta,\varphi)$ given by \eqref{estim} can only increase the upper bound when $x > K_1.$
\end{proof}

Now, let us estimate of probabilities of large deviations between extremes of the solution field $u(\eta, \theta, \varphi)$ and its approximation $u_L(\eta, \theta, \varphi).$ The truncation error field $\widehat{u}_L(\eta, \theta, \varphi) = {u}(\eta, \theta, \varphi) -  u_L(\eta, \theta, \varphi), \ L \in \mathbb{N},$ is a centered Gaussian random field. Analogously to the considered results, it generates the canonic pseudometric $d_{L,\eta}(\Theta)$ on $\mathbb{S}^2$ and the function $g_{L,\eta}(\varepsilon)$ is associated to $d_{L,\eta}(\cdot),$ which is given by $g_{L,\eta}(\varepsilon):=\min\{\Theta: d_{L,\eta}(\Theta) \geq \varepsilon\}.$ Analogously to the previous results one obtains:

\begin{corollary}
For each fixed $\eta\in[0,\eta_\infty)$ the following estimate holds true
\begin{equation}\label{cor:estim}
P\left( \sup_{\theta,\varphi}|\widehat{u}_L(\eta,\theta,\varphi)| > x \right) \leq 2\exp\left(-\frac{\left(x-E\left(\sup_{\theta,\varphi}\widehat{u}_L(\eta,\theta,\varphi)\right)\right)^2}{2\sigma_{\eta,L}^2}\right)
\end{equation}
where $\sigma_{\eta,L}^2 = (4\pi)^{-1}\sum_{l=L+1}^\infty   C_l  (2l+1) F^2_l(\eta).$

If there exists $\beta \in (0,2],$ such that
\begin{equation}\label{cor1:cond} \sum\limits_{l=0}^\infty C_ll^{1+\beta} <+ \infty,\end{equation} then
\[P\left(  \sup\limits_{\theta,\varphi}|\widehat{u}_L(\eta,\theta,\varphi)| > x \right) \leq 2\exp\left( -\frac{\left(x- K_2\right)^2}{2\sigma_{\eta,L}^2}
    \right)
    \]
     for $x> K_2:=K \int\limits_0^R \sqrt{\log \left( \frac{2}{1-\cos(g_{L,\eta}(\varepsilon))}\right)}d\varepsilon,$ where the latter integral is finite.

\begin{proof} As the random field $\widehat{u}_L(\eta,\theta, \varphi)$ is centered Gaussian, it holds for $x>0$ that
\[ P\left( \sup_{\theta,\varphi}|\widehat{u}_L(\eta,\theta, \varphi)| > x \right)  \leq 2P\left( \sup_{\theta,\varphi}\widehat{u}_L(\eta,\theta, \varphi) > x \right). \]
By
\[ E \widehat{u}^2_L(\eta,\theta, \varphi) =  (4\pi)^{-1}\sum_{l=L+1}^\infty   C_l  (2l+1) F^2_l(\eta),\]  the inequality in \eqref{cor:estim} is a direct corollary of Theorem~\ref{extreme}.

The random field $\widehat{u}_L(\eta,\theta, \varphi)$ is $L_2$-continuous in $\theta$ and $\varphi$ by condition \eqref{cor1:cond}. By applying Fernique's inequality with the normalized uniform distribution $\mu$ one obtains

\[
 E\sup\limits_{\theta,\varphi}\widehat{u}_L(\eta, \theta, \varphi) \leq K \int\limits_0^R \sqrt{\log \left( \frac{2}{1-\cos(g_{L,\eta}(\varepsilon))}\right)}d\varepsilon.
\] The convergence of the above integral can be shown by using steps analogous to the proof of Theorem~\ref{lm:finit}. The application of the above estimate to \eqref{cor:estim} completes the proof.
\end{proof}

 \end{corollary}

\section{Numerical studies}\label{sec_num}

This section presents numerical studies of the solution $u(\eta,\theta,\varphi)$ of the initial value problem \eqref{eq_main}, \eqref{cond_rand1} and \eqref{cond_rand2} and its truncated approximation $u_L(\eta, \theta, \varphi)$ from Section \ref{sec:approx}.

Numerical computations  were performed using the software R version 4.3.1 and Python version 3.11.5. The HEALPix representation (\mbox{\url{http://healpix.sourceforge.net}}) and the Python package "healpy" were used for computations. The Python package "pyshtools" was used to simulate spherical random fields.  The R package "rcosmo", see \cite{fryer2019rcosmo} and \cite{fryer2018rcosmo}, was used for visualisations. The R and Python code are freely available in the folder "Research materials" from the  website \mbox{\url{https://sites.google.com/site/olenkoandriy/}}. For the numerical studies, we use a map of CMB intensities obtained by the ESA mission Planck, these data are available in \cite{link1} as well as its angular power spectrum in \cite{link2}.

\subsection{Functions \texorpdfstring{$F_l(\eta)$}{F\_l}}

The evolution in time of the solution $u(\eta,\theta,\varphi)$ is governed by the functions $F_l(\eta),$ $\eta\in[0,\eta_\infty),$ $l=0,1,2,...,$ given by \eqref{eq:fl}. These functions are also the multiplication factors defining the change of the angular power spectrum of the solution with time $\eta$. Lemma \ref{lemma2.1} showed that for a fixed $\eta$ the multiplication factors have a wave structure for large $l.$ Figure \ref{fig:fl1} depicts multiplication factors $F_\eta(l)$ as functions of the argument $l$ for the fixed time moments $\eta=0.001$ and $\eta=0.002.$  One can see that $F_\eta(l),$ as the functions of the argument $l,$ form waves whose periods decrease in time. On the other hand,  Figure \ref{fig:fl2} displays $F_\eta(l)$ as functions of the argument $\eta$ for the fixed values of indices $l=3$ and $l=10.$ In this case, the functions $F_\eta(l)$ form waves whose amplitudes decrease  as they propagate. Figure \ref{fig:fl2} also demonstrates that their periods decrease with $l.$

\begin{figure}[htb!]
\centering
\begin{subfigure}{.45\textwidth}
  \centering
  \includegraphics[width=\linewidth, height=0.7\linewidth]{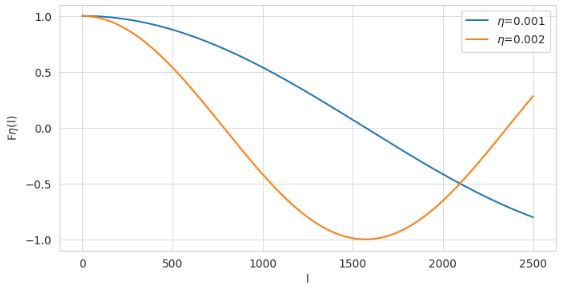}    \caption{$F_\eta(l)$ as  functions of $l$ for fixed $\eta=0.001$ and $\eta=0.002$ }
  \label{fig:fl1}
\end{subfigure}\hspace{5mm}
\begin{subfigure}{.45\textwidth}
  \centering
  \includegraphics[width=\linewidth, height=0.7\linewidth]{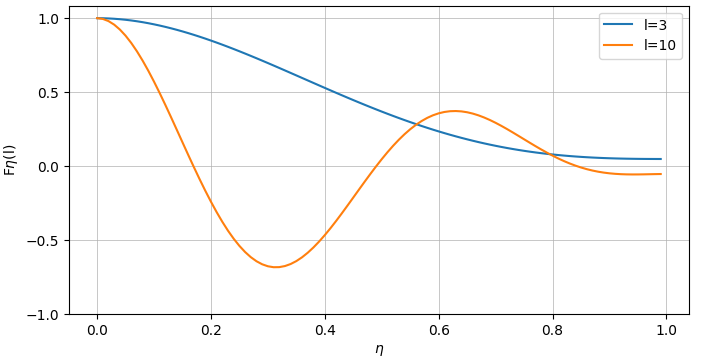} \caption{$F_\eta(l)$ as functions of $\eta$ for fixed $l=3$ and $l = 10$}
  \label{fig:fl2}
\end{subfigure}
\caption{$F_\eta(l)$ as functions of arguments $l$ and $\eta$}
\label{fig:fl}
\end{figure}

\subsection{Initial condition}

In this section numerical studies  of the solution of the initial value problem \eqref{eq_main}, \eqref{cond_rand1} and \eqref{cond_rand2} use the initial condition field consisting of the measurements of CMB intensities.  From the mathematical point of view, these values of the CMB intensities form a single realization of a spherical isotropic Gaussian random field. Thus, its spherical harmonics representation \eqref{eq:repres} can be derived. The Python's package "healpy" was used to numerically calculate values of the coefficients $a_{lm}$ from the map of CMB intensities. In the following numerical examples the values $a_{lm}$ for $l=0,1,...,2500,$ and $m=-l,...,l$ and the corresponding  initial condition  $u(0,\theta,\varphi) =\sum_{l=0}^{2500}\sum_{m=-l}^l a_{lm} Y_{lm}(\theta, \varphi)$ were used.

\subsection{Evolution of the solution}\label{sub:evol}
This section provides numerical simulations for the evolution in space and time according to the model \eqref{eq_main}, \eqref{cond_rand1} and \eqref{cond_rand2}, studies spatio-temporal dependencies, evolution of the angular spectrum, and the behaviour of the corresponding extremes. For illustrative purposes, the constants $c,D,r,\eta_\infty$ were  set equal to $1.$

The structure of the CMB data is quite complex with various hot and cold areas located close to each other. Figure \ref{fig:fig1} shows the map of CMB, used as the initial condition for the considered problem \eqref{eq_main} \eqref{cond_rand1}, and \eqref{cond_rand2}.  The representation \eqref{sol_rand} was used to obtain the solution $u(\eta, \theta, \varphi)$ for $\eta=0.001.$  From the  visualisation of $u(0.001,\theta,\varphi)$ in  Figure~\ref{fig:fig2}, one can see that the solution becomes smoother and large deviations (both positive and negative) from the overall mean  are decreasing in time.

\begin{figure}[htb!]
\centering
\begin{subfigure}{.45\textwidth}
  \centering
  \includegraphics[width=0.9\linewidth, height=0.75\linewidth]{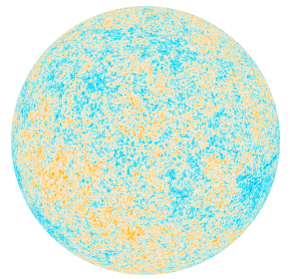}
  \caption{Case of $\eta = 0$}
  \label{fig:fig1}
\end{subfigure}%
\begin{subfigure}{.45\textwidth}
  \centering
  \includegraphics[width=0.9\linewidth, height=0.75\linewidth]{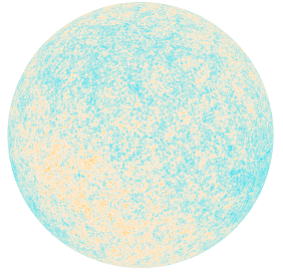}
  \caption{Case of $\eta = 0.001$}
  \label{fig:fig2}
\end{subfigure}
\caption{Realizations of $u(\eta,\theta,\varphi)$.}
\label{fig:test}
\end{figure}

\begin{figure}[htb!]
\centering
\begin{subfigure}{.5\textwidth}
  \centering
  \includegraphics[width=\linewidth, height=0.6\linewidth]{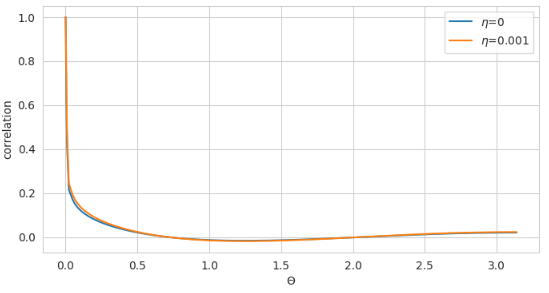}
  \caption{Spatial correlations for $\eta=0$ and $\eta=0.001$}
  \label{fig:correl1}
\end{subfigure}%
\begin{subfigure}{.5\textwidth}
  \centering
  \includegraphics[width=\linewidth, height=0.6\linewidth]{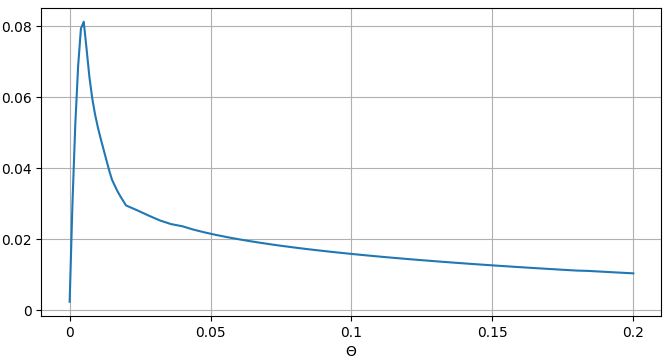}
  \caption{Difference between correlations in Fig \ref{fig:correl1}}
  \label{fig:correl2}
\end{subfigure}
\vskip\baselineskip

\begin{subfigure}{.5\textwidth}
  \centering
  \includegraphics[width=\linewidth, height=0.6\linewidth]{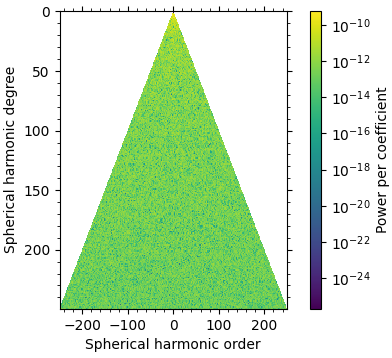}
  \caption{Coefficients $a_{lm}$ for $\eta=0$}
  \label{fig:alm1}
\end{subfigure}%
\begin{subfigure}{.5\textwidth}
  \centering
  \includegraphics[width=\linewidth, height=0.6\linewidth]{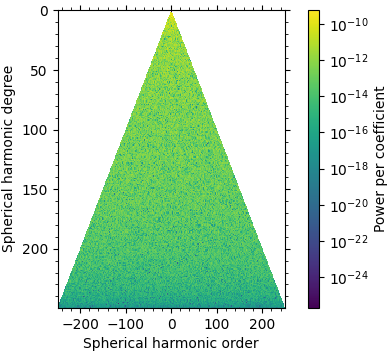}
  \caption{Coefficients $a_{lm}$ for $\eta=0.001$}
  \label{fig:alm2}
\end{subfigure}
\caption{Spatial dependencies and coefficients $a_{lm}$ for $\eta=0$ and $\eta=0.001.$}
\label{fig:correl}
\end{figure}
 Figure \ref{fig:correl1} demonstrates how spatial dependencies change with increasing angular distance between spherical locations. The spatial correlation function decreases rapidly with the increase of the angular distance $\Theta.$ Figures \ref{fig:correl1} and \ref{fig:correl2} show the change of spatial correlations for $\eta=0$ and $\eta=0.001.$

 Figures \ref{fig:alm1} and \ref{fig:alm2} plot values of the spherical harmonic coefficients $a_{lm}$  of the realizations in Figure~\ref{fig:test} for levels up to 200.  It is evident from large levels $l,$ that for $\eta=0.001$ the coefficients $a_{lm}$  became smaller compared to the corresponding coefficients for the case $\eta=0.$ The plots suggest that the magnitudes of the coefficients $a_{lm}$ are decreasing with $\eta.$
\begin{figure}[htb!]
    \centering
    \includegraphics[width=0.65\linewidth, height=0.3\linewidth]{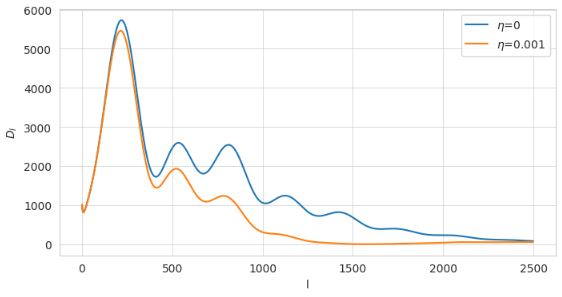}    \caption{Angular spectra of $u(\eta,\theta,\varphi)$ for $\eta=0$ and $\eta=0.001$}
    \label{fig:angspec}
\end{figure}

We also employed the angular power spectrum $C_l$ of CMB provided by ESA \cite{link2}. It is depicted  in Figure~\ref{fig:angspec}  when $\eta=0$.  Then, the spherical angular power spectrum  for the case of $\eta=0.001$ was computed.  The plots of spherical angular power spectra  for $\eta=0$ and $\eta=0.001$  show that they remain almost identical for small values of $l$ and flatten out when $l$ and $\eta$ increase.  This observation aligns with the obtained theoretical findings  and cosmological theories positing that higher multipoles undergo more rapid changes.

To illustrate the structure of space-time dependencies, we also produced a 3D-plot showing normalized covariances (divided by $Eu^2(0,0,0))$  as a function of angular distance $\Theta$ and time~$\eta.$ In Figure~\ref{fig:corspacetimel}, when $\Theta$ and $\eta$ increase, the normalized covariance function rapidly decreases for small values of $\Theta$ and $\eta$ in both space and time; whereas, for other values, the decay of covariances is rather slow. One can see that the variance of the solution $Eu^2(\eta,0,0)$ is decaying with $\eta$ rapidly for small values of $\eta.$

\begin{figure}[htb!]
    \centering
    \includegraphics[width=0.55\linewidth, height=0.4\linewidth]{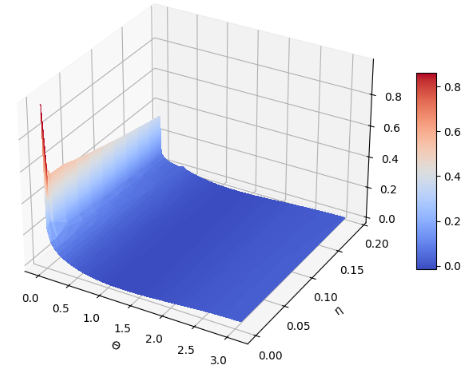}    \caption{Spatio-temporal covariances of $u(\eta,\theta,\varphi).$}
    \label{fig:corspacetimel}
\end{figure}

\subsection{Excursion probabilities}

To analyse the excursion probabilities we used the solutions with the initial value field possessing the power spectrum shown in Figure \ref{fig:angspec} for the case $\eta=0.$ By setting the value $\eta=0.001$ and by using the relationships \eqref{eq:met} and \eqref{geps},  one can compute the canonical pseudometric $d_{0.001}(\Theta)$  generated by $u(0.001,\theta,\varphi)$ on the sphere, along with the associated function $g_{0.001}(\varepsilon).$ The results of these computations are illustrated  in Figure \ref{fig:metr}. The plots  demonstrate a rapid increase of the function $d_{0.001}(\Theta)$ in the neighbourhood of the origin. This observed behaviour of the canonical pseudometric can be attributed to  the rapid change of the spatial covariance function $K_{0.001}(\Theta)$ at the origin, see Figure~\ref{fig:correl1} and consult  \eqref{eq:met} on their relations.

\begin{figure}[htb!]
\centering
\begin{subfigure}{.4\textwidth}
  \centering
  \includegraphics[width=\linewidth, height=0.7\linewidth]{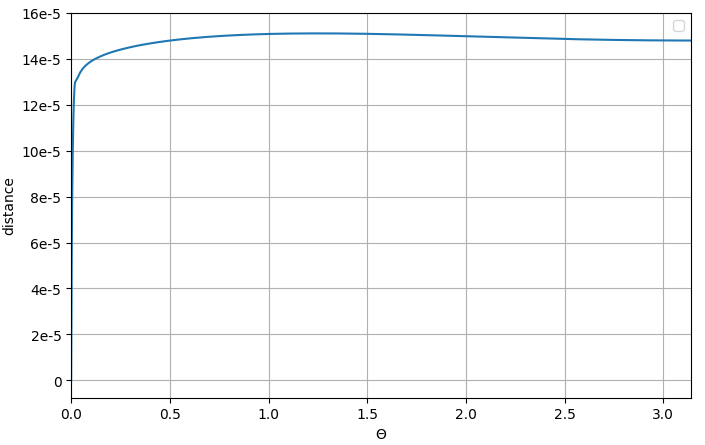}
  \caption{Canonical pseudometric $d_\eta(\Theta)$}
  \label{fig:metr1}
\end{subfigure}%
\begin{subfigure}{.4\textwidth}
  \centering
  \includegraphics[width=\linewidth, height=0.7\linewidth]{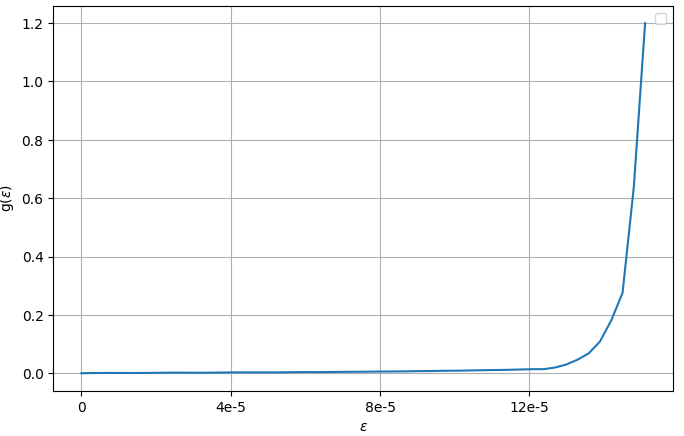}
  \caption{The function $g_{\eta}(\varepsilon)$}
  \label{fig:metr2}
\end{subfigure}
\caption{Canonical pseudometric $d_\eta(\Theta)$ and the function $g_{\eta}(\varepsilon)$ for $\eta=0.001$.}
\label{fig:metr}
\end{figure}

By using the Python package "pyshtools" and the values $C_l$ of the  angular power spectrum provided by ESA, we simulate 300 realizations of spherical Gaussian random fields. Then, those realizations were used as the initial conditions for the problem \eqref{eq_main}, \eqref{cond_rand1}, and \eqref{cond_rand2}. By decomposing each of these 300 realizations into a series of spherical harmonics and by applying the formula \eqref{sol_rand} with $\eta=0.001$, we obtained 300 realizations of the solution for $\eta=0.001.$

\begin{figure}[htb!]
    \centering
    \includegraphics[width=0.7\linewidth, height=0.28\linewidth]{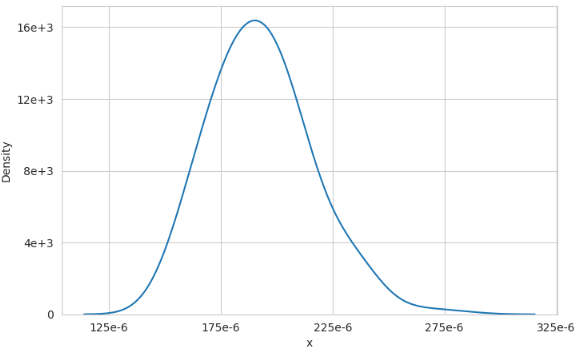}    \caption{Sample distribution of $\sup_{\theta,\varphi}u(0.001, \theta,\varphi)$.}
    \label{fig:dist_sup}
\end{figure}

\begin{figure}[htb!]
    \centering
    \includegraphics[width=0.7\linewidth, height=0.28\linewidth]{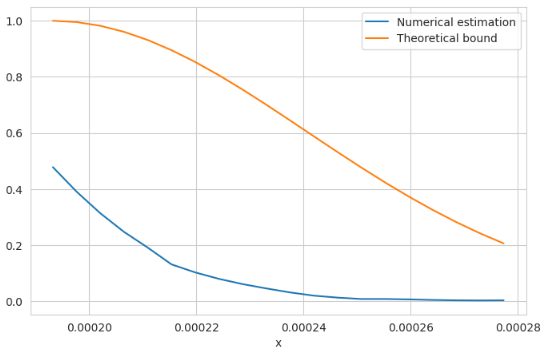}    \caption{Sample probabilities $P(\sup_{\theta,\varphi}u(0.001, \theta,\varphi)>x)$ and theoretical bounds.}
    \label{fig:excur}
\end{figure}

Figure \ref{fig:dist_sup} shows the sample distribution of $\sup_{\theta,\varphi}u(0.001, \theta,\varphi)$ obtained by using the  realizations of the solution. It is known that the distribution of the supremum of a Gaussian process on an interval is positively skewed as its median is less than or equal to its mean, see \cite[Section 12]{lifshits1995gaussian}. The similar behaviour for the case of spherical Gaussian random fields is observed in Figure \ref{fig:dist_sup}. Figure~\ref{fig:excur} provides the estimated excursion probabilities $P(\sup_{\theta,\varphi}u(0.001, \theta,\varphi)>x)$ and the corresponding upper bound given in Theorem~\ref{th:excur}. The observed pattern is analogous for this type of estimators. For small values of $x,$ the upper bound appears conservative,  however, it rapidly approaches the excursion probability as $x$ increases.   Figure~\ref{fig:excur} illustrates the areas  for potential improving of this bound.

\section{Conclusions and future work} The paper investigated spherical diffusion within an expanding space-time framework. It derived  solutions to the deterministic and stochastic diffusion equations, specifically focusing on the exponential growth case. The findings offer insights into the asymptotic and local characteristics of the solutions. Additionally, the paper explored the extremal properties of these solutions. Given the validity of Borell-TIS and Fernique's inequalities for general sets, a similar methodology can be employed to investigate extremes of  solutions on more complex manifolds. Applications of the theoretical findings were demonstrated through simulation studies by modelling  evolutionary scenarios and examining properties of the obtained estimates with the CMB intensities as the initial condition. In future research, it would be interesting to further study properties of this model and the corresponding solutions and to extend the developed approach  to  other SPDEs.

\section*{Compliance with ethical standards}
\textbf{Conflict of interest} The authors declare that there is no conflict of interest regarding the publication of
the paper. \\
\textbf{Funding} This research was supported under the Australian Research Council's Discovery Projects funding scheme (project number  DP220101680).  I.Donhauzer and and A.Olenko also would like to thank for partial support provided by the La Trobe SEMS CaRE grant.


\begin{thebibliography}{10}

\bibitem{link1}
\url{http://irsa.ipac.caltech.edu/data/Planck/release_2/all-sky-maps/maps/component-maps/cmb/COM_CMB_IQU-smica_1024_R2.02_full.fits}.

\bibitem{link2}
\url{https://irsa.ipac.caltech.edu/data/Planck/release_2/ancillary-data/previews/ps_index.html}.

\bibitem{abramowitz1968handbook}
M.~Abramowitz and I.~Stegun.
\newblock {\em Handbook of Mathematical Functions with Formulas, Graphs, and
  Mathematical Tables}.
\newblock US Government printing office, Washington, 1968.

\bibitem{adam2016planck}
R.~Adam, P.~Ade, N.~Aghanim, Y.~Akrami, M.~Alves, F.~Arg{\"u}eso, M.~Arnaud,
  F.~Arroja, M.~Ashdown, J.~Aumont, et~al.
\newblock Planck 2015 results-{I}. {Overview of products and scientific
  results}.
\newblock {\em Astron. Astrophys.}, 594:A1, 2016.

\bibitem{ade2016planck}
P.~Ade, N.~Aghanim, Y.~Akrami, P.~Aluri, M.~Arnaud, M.~Ashdown, J.~Aumont,
  C.~Baccigalupi, A.~J. Banday, R.~Barreiro, et~al.
\newblock Planck 2015 results-{XVI. Isotropy and statistics of the CMB}.
\newblock {\em Astron. Astrophys.}, 594:A16, 2016.

\bibitem{adler2007random}
R.~Adler and J.~Taylor.
\newblock {\em {Random Fields and Geometry}}.
\newblock Springer, New York, 2007.

\bibitem{anh2018approximation}
V.~Anh, P.~Broadbridge, A.~Olenko, and Y.~Wang.
\newblock On approximation for fractional stochastic partial differential
  equations on the sphere.
\newblock {\em Stoch. Environ. Res. Risk Assess.}, 32:2585--2603, 2018.

\bibitem{arfken1972mathematical}
G.~Arfken and H.~Weber.
\newblock {\em {Mathematical Methods for Physicists}}.
\newblock Academic Press, New York, 1972.

\bibitem{atkinson2012spherical}
K.~Atkinson and W.~Han.
\newblock {\em {Spherical Harmonics and Approximations on the Unit Sphere: An
  Introduction}}.
\newblock Springer Science, New York, 2012.

\bibitem{broadbridge2023non}
P.~Broadbridge, S.~Becirevic, and D.~Hoxley.
\newblock {\em Non-Particulate Quantum States of the Electromagnetic Field in
  Expanding Space-Time}.
\newblock IntechOpen, 2023.

\bibitem{broadbridge2020solution}
P.~Broadbridge and K.~Deutscher.
\newblock {Solution of non-autonomous Schr{\"o}dinger equation for quantized de
  Sitter Klein-Gordon oscillator modes undergoing attraction-repulsion
  transition}.
\newblock {\em Symmetry}, 12(6):943, 2020.

\bibitem{broadbridge2019random}
P.~Broadbridge, A.~D. Kolesnik, N.~Leonenko, and A.~Olenko.
\newblock Random spherical hyperbolic diffusion.
\newblock {\em J. Stat. Phys.}, 177:889--916, 2019.

\bibitem{broadbridge2020spherically}
P.~Broadbridge, A.~D. Kolesnik, N.~Leonenko, A.~Olenko, and D.~Omari.
\newblock Spherically restricted random hyperbolic diffusion.
\newblock {\em Entropy}, 22(2):217, 2020.

\bibitem{buldygin2000metric}
V.~Buldygin and Y.~Kozachenko.
\newblock {\em {Metric Characterization of Random Variables and Random
  Processes}}.
\newblock American Mathematical Soc., Providence, 2000.

\bibitem{buonocore2021anomalous}
S.~Buonocore and M.~Sen.
\newblock Anomalous diffusion of cosmic rays: A geometric approach.
\newblock {\em AIP Adv.}, 11(5):055221.

\bibitem{cheng2016excursion}
D.~Cheng and Y.~Xiao.
\newblock {Excursion probability of Gaussian random fields on sphere}.
\newblock {\em Bernoulli}, 22(2):1113--1130, 2016.

\bibitem{dudley1967sizes}
R.~Dudley.
\newblock {The sizes of compact subsets of Hilbert space and continuity of
  Gaussian processes}.
\newblock {\em J. Funct. Anal.}, 1(3):290--330, 1967.

\bibitem{dunkel2009relativistic}
J.~Dunkel and P.~H{\"a}nggi.
\newblock {Relativistic Brownian motion}.
\newblock {\em Phys. Rep.}, 471(1):1--73, 2009.

\bibitem{d2014coordinates}
M.~D’Ovidio.
\newblock Coordinates changed random fields on the sphere.
\newblock {\em Journal of Statistical Physics}, 154:1153--1176, 2014.

\bibitem{fryer2019rcosmo}
D.~Fryer, M.~Li, and A.~Olenko.
\newblock {rcosmo: R package for analysis of spherical, HEALPix and
  cosmological data}.
\newblock {\em {R J.}}, 12(1):206--225, 2020.

\bibitem{fryer2018rcosmo}
D.~Fryer, A.~Olenko, M.~Li, and Y.~Wang.
\newblock rcosmo: Cosmic microwave background data analysis. {R} package
  version 1.1.3.
\newblock Available at CRAN, 2021.

\bibitem{han2017observing}
J.~Han.
\newblock Observing interstellar and intergalactic magnetic fields.
\newblock {\em Annu. Rev. Astron. Astrophys.}, 55:111--157, 2017.

\bibitem{hill2018formal}
J.~Hill.
\newblock On the formal origins of dark energy.
\newblock {\em Zeitschrift f{\"u}r angewandte Mathematik und Physik},
  69(5):133, 2018.

\bibitem{hill2019some}
J.~Hill.
\newblock Some further comments on special relativity and dark energy.
\newblock {\em Zeitschrift f{\"u}r angewandte Mathematik und Physik}, 70(1):5,
  2019.

\bibitem{jokipii1969stochastic}
J.~Jokipii and E.~Parker.
\newblock Stochastic aspects of magnetic lines of force with application to
  cosmic-ray propagation.
\newblock {\em Astrophys. J.}, 155:777--799, 1969.

\bibitem{lang2015isotropic}
A.~Lang and C.~Schwab.
\newblock Isotropic {G}aussian random fields on the sphere: regularity, fast
  simulation and stochastic partial differential equations.
\newblock {\em Ann. Appl. Probab.}, 25:3047--3094, 2015.

\bibitem{LOV}
N.~Leonenko, A.~Olenko, and J.~Vaz.
\newblock On fractional spherically restricted hyperbolic diffusion random
  field.
\newblock {\em arXiv: 2310.03933}, 2023.

\bibitem{leonenko_ruiz-medina_2023}
N.~Leonenko and M.~Ruiz-Medina.
\newblock {Sojourn functionals for spatiotemporal Gaussian random fields with
  long memory}.
\newblock {\em J. Appl. Probab.}, 60(1):148–165, 2023.

\bibitem{lifshits1995gaussian}
M.~Lifshits.
\newblock {\em {Gaussian Random Functions}}.
\newblock Springer, Dordrecht, 1995.

\bibitem{makogin}
V.~Makogin and E.~Spodarev.
\newblock {Limit theorems for excursion sets of subordinated Gaussian random
  fields with long-range dependence}.
\newblock {\em Stochastics}, 94(1):111--142, 2022.

\bibitem{marinucci2011random}
D.~Marinucci and G.~Peccati.
\newblock {\em {Random Fields on the Sphere: Representation, Limit Theorems and
  Cosmological Applications}}.
\newblock Cambridge University Press, 2011.

\bibitem{olenko2006upper}
A.~Olenko.
\newblock Upper bound on $\sqrt{x}$${J}_\nu(x)$ and its applications.
\newblock {\em Integral Transforms Spec. Funct.}, 17(6):455--467, 2006.

\bibitem{DLMF}
F.~Olver, A.~O. Daalhuis, D.~Lozier, B.~Schneider, R.~Boisvert, C.~Clark, B.~S.
  B.~Mille~and, H.~Cohl, and {M. McClain, eds.}
\newblock {NIST Digital Library of Mathematical Functions}, 2020.
\newblock \href{http://dlmf.nist.gov/}{http://dlmf.nist.gov/}.

\bibitem{orsingher1987stochastic}
E.~Orsingher.
\newblock Stochastic motions on the 3-sphere governed by wave and heat
  equations.
\newblock {\em J. Appl. Probab.}, 24(2):315--327, 1987.

\bibitem{piterbarg1996asymptotic}
V.~Piterbarg.
\newblock {\em {Asymptotic Methods in the Theory of Gaussian Processes and
  Fields}}.
\newblock American Mathematical Society, Providence, 1996.

\bibitem{sakhno2023estimates}
L.~Sakhno.
\newblock Estimates for distributions of suprema of spherical random fields.
\newblock {\em Stat. Optim. Inf. Comput.}, 11(2):186--195, 2023.

\bibitem{shtykovskiy2010thermal}
P.~Shtykovskiy and M.~Gilfanov.
\newblock Thermal diffusion in the intergalactic medium of clusters of
  galaxies.
\newblock {\em Mon. Notices Royal Astron. Soc.}, 401(2):1360--1368, 2010.

\bibitem{talagrand}
M.~Talagrand.
\newblock {Sharper Bounds for Gaussian and Empirical Processes}.
\newblock {\em Ann. Probab.}, 22(1):28 -- 76, 1994.

\bibitem{talagrand1996majorizing}
M.~Talagrand.
\newblock Majorizing measures: the generic chaining.
\newblock {\em Ann. Probab.}, 24(3):1049--1103, 1996.

\bibitem{talagrand2014upper}
M.~Talagrand.
\newblock {\em {Upper and Lower Bounds for Stochastic Processes}}.
\newblock Springer, 2014.

\bibitem{watson1944}
G.~Watson.
\newblock {\em {A Treatise on the Theory of Bessel Functions}}.
\newblock Cambridge University Press, 1944.

\bibitem{Weinberg}
S.~Weinberg.
\newblock {\em Cosmology}.
\newblock {Oxford University Press}, 2008.

\end{thebibliography}
\end{document}